\documentclass[12pt,final
]{amsart}

\usepackage{amsfonts,amssymb,amsmath}
\usepackage{url,color}
\usepackage{enumerate}
\usepackage[comma, authoryear]{natbib}
\usepackage[colorlinks]{hyperref}

\usepackage{graphicx}
\newtheorem{theorem}{Theorem}[section]
\newtheorem{lemma}{Lemma}[section]
\newtheorem{proposition}{Proposition}[section]
\newtheorem{corollary}{Corollary}[section]

\newtheorem{remark}{Remark}[section]
\newtheorem{example}{Example}[section]
\newtheorem{hyp}{Hypothesis}[section]
\numberwithin{equation}{section}

\newcommand{\I}{\ensuremath{\mathbb{I}}}

\newcommand{\dR}{\ensuremath{\mathbb{R}}}

\newcommand{\dZ}{\ensuremath{\mathbb{Z}}}

\newcommand{\cA}{\ensuremath{\mathcal{A}}}

\newcommand{\cD}{\ensuremath{\mathcal{D}}}
\newcommand{\cE}{\ensuremath{\mathcal{E}}}

\newcommand{\cM}{\ensuremath{\mathcal{M}}}
\newcommand{\cN}{\ensuremath{\mathcal{N}}}

\newcommand{\cU}{\ensuremath{\mathcal{U}}}
\newcommand{\cV}{\ensuremath{\mathcal{V}}}

\newcommand{\inw}[1]{\stackrel{\circ}{#1}}

\def\qed{\hfill\rule{2mm}{2mm}}

\def\disp{\displaystyle}

\def\un{\left[\begin{array}{c} \mathbf{1} \\
\mathbf{1}\end{array}\right]}

\usepackage[color, notref, notcite]{showkeys}
\usepackage{enumerate}
\definecolor{labelkey}{rgb}{0.6,0,1}

\begin{document}

\title{Combining losing games into a winning game}

\author{Bruno R\'{e}millard}

\address{CRM, GERAD and Department of Decision Sciences, HEC Montr\'{e}al,\\
3000, che\-min de la C\^{o}\-te-Sain\-te-Ca\-the\-ri\-ne,
Montr\'{e}al (Qu\'{e}\-bec), Canada H3T 2A7}
\email{bruno.remillard@hec.ca}

\author{Jean Vaillancourt}
\address{Department of Decision Sciences, HEC Montr\'{e}al,\\
3000, che\-min de la C\^{o}\-te-Sain\-te-Ca\-the\-ri\-ne,
Montr\'{e}al (Qu\'{e}\-bec), Canada H3T 2A7} \email{jean.vaillancourt@hec.ca}
\thanks{Partial funding in support of this work was provided by the Natural
Sciences and Engineering Research Council of Canada and  the Fonds
qu\'e\-b\'e\-cois de la re\-cher\-che sur la na\-tu\-re et les
tech\-no\-lo\-gies.}

\begin{abstract}
Parrondo's
paradox is extended to
regime switching random walks in random environments.
The paradoxical behavior of the resulting random walk is explained
by the effect of the random environment. Full characterization of the asymptotic behavior is achieved in terms of  the dimensions of some random subspaces occurring in Oseledec's theorem. The regime switching mechanism gives our models a richer and more complex asymptotic behavior than the simple random walks in random environments appearing in the literature, in terms of transience and recurrence.
\end{abstract}




\maketitle

\section{Introduction}

To illustrate what is now known as Parrondo's paradox
\citep{Harmer/Abbott:1999}, consider the following games:
\begin{itemize}
\item Game A: The fortune $X_n$ of the player after $n$ independent games is
$$
X_{n} =  \left\{\begin{array}{ll} X_{n-1} + 1 & \mbox{ w. pr. }  p, \\
X_{n-1} - 1 & \mbox{ w. pr. }  1- p, \end{array}\right., \quad n\ge
1.
$$

\item Game B:  The fortune $Y_n$ of the player after $n$ games is
given by
$$
Y_{n} =  \left\{\begin{array}{lll}
 Y_{n-1}+ 1 & \mbox{ w. pr. }  g(Y_{n-1}), \\Y_{n-1} - 1 & \mbox{ w. pr. }  1- g(Y_{n-1}),
 \end{array}\right. , \quad n\ge 1,
$$
where  $g$ is a 3-periodic function on $\dZ$ such that $g(0) = p_1$ and
$g(1)=g(2)=p_2$.
\end{itemize}
It is well-known that Game A is fair if and only if $p=1/2$. In
fact, if $p\neq 1/2$, the Markov chain $X$ is transient and
 $\lim_{n\to\infty} X_n = -\infty$ if $p < 1/2$, while  $\lim_{n\to\infty} X_n = +\infty$ if $p > 1/2$. When $p=1/2$, $X$ is recurrent and
$$
 P(\liminf X_n = -\infty  \mbox{ and } \limsup X_n = +\infty)=1.
 $$

For Game B, the process $Y$ can be seen as a particular case of a
random walk in a random environment; in fact, the space $E$ of environments has $3$ elements, i.e., $E =\{T^k g; k=0,1,2\}$ with $T^k g(x) = g(x+k)$. For more details on periodic and almost periodic environments, see, e.g.,  \citet[Examples 1-2]{Remillard/Dawson:1989}.

\citet{Solomon:1975}
studied the special case  behavior of random walks in a random environment when the latter are
independent and identically distributed (i.i.d.) which does not cover the periodic environment. Instead, one can rely on  \citet[Theorem
2.1]{Alili:1999}, to conclude that the process $Y$ is recurrent
if and only if $\mu=1$, where
\begin{equation}\label{eq:mu}
 \mu = \frac{(1-p_1)(1-p_2)^2}{p_1 p_2^2}.
 \end{equation}
As a result, $
P\left( \liminf Y_n = -\infty  \mbox{ and } \limsup Y_n = +\infty\right)=1$.
Otherwise, when $\mu \neq 1$, $Y_n$ is transient and
$\lim_{n\to\infty} Y_n = -\infty$ if $\mu > 1$,  while
$\lim_{n\to\infty} Y_n = +\infty$ if $\mu <1$. For example, if $p_1=
1/10$ and $p_2 = 3/4$, then $\mu=1$. If $p_1 < 1/10$ and $p_2 <
3/4$, then $\mu>1$.

Following \citet{Harmer/Abbott:1999}, suppose that $p=0.499$, $p_1 = 0.099$ and $p_2 =
0.749$. Then,
according to the previous observations, if a player always plays
Game A or Game B, her fortune will tend to $-\infty$ with
probability one. However, if she plays Game A twice, then Game B
twice and so on (Game C),  or if she chooses the game at random with
probability $1/2$ (Game D), her fortune will tend to $+\infty$. This
is Parrondo's paradox and it is illustrated in Figure
\ref{fig:parrondo}.

\begin{figure}[h!]
\begin{center}
\includegraphics[scale = 0.35]{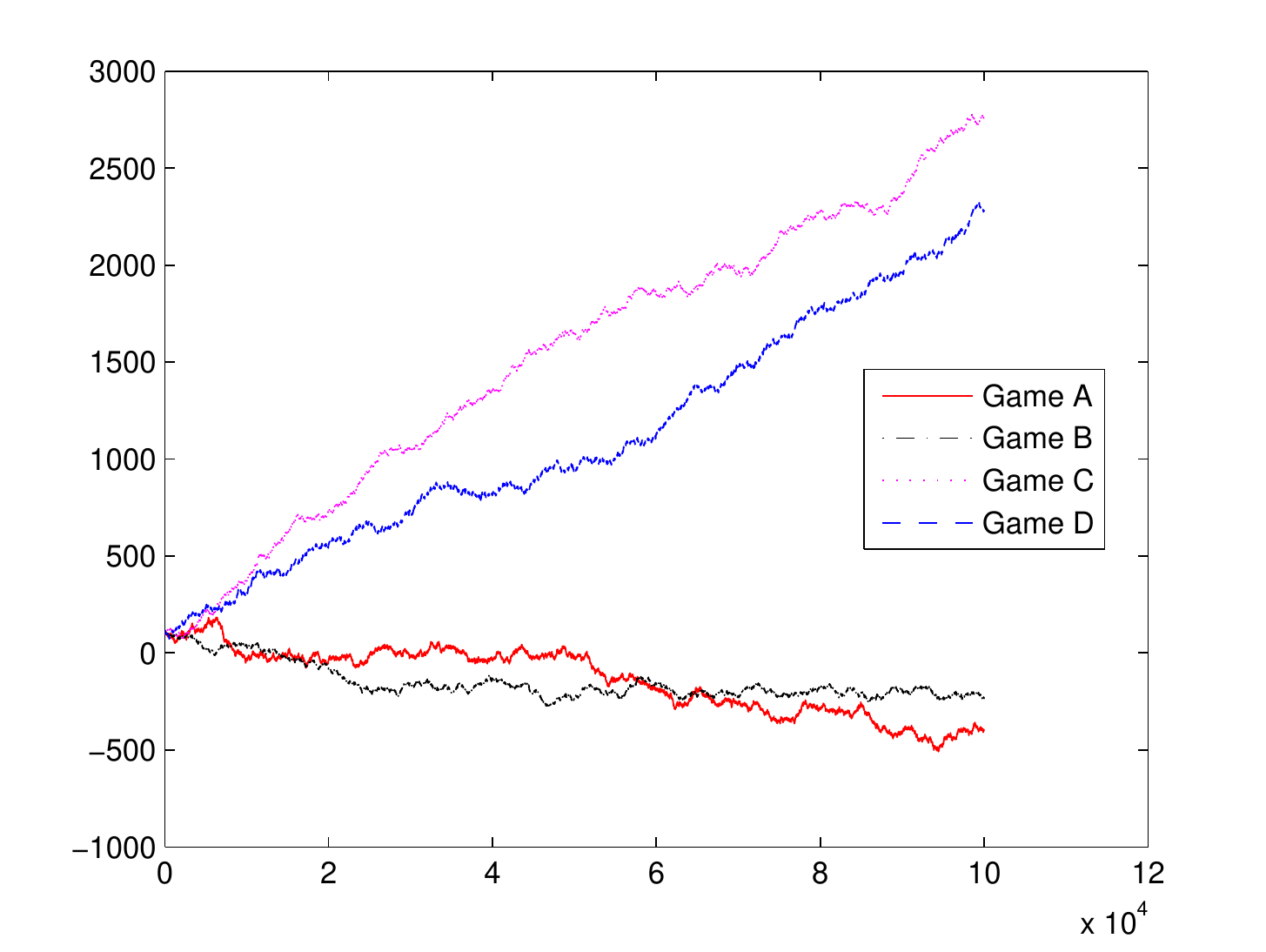}
\caption{Evolution of the fortune of a player starting at $100\$$
over 10,000 games, where $p=0.499$, $p_1 = 0.099$ and $p_2 =
0.749$.}\label{fig:parrondo}
\end{center}
\end{figure}

Now consider Game C' , where the player alternates between Game A and Game B, i.e.,  she plays Game A once, then Game B once, and so on. What happens in this case? The answer will be given at the end of Example \ref{ex:periodic2}.

 Games A and B are particular cases of random walks in a
random environment, while Games C and D are examples of regime
switching Markov chains in random environment. The aim of this paper
is to study the asymptotic behavior of the latter. One of the first rigorous
work on this problem is \citet{Pyke:2003}, who studied some particular
cases of random environment, namely the so-called periodic case,
where each random walk is like in Game B, while the player chooses
at random between two games. The author also consider some
``deterministic'' mixtures, namely the cases studied in
\citet{Harmer/Abbott:1999}.
The earliest examples of games exhibiting this paradoxical behavior when combined can be found in
\citet{Durrett/Kesten/Lawler:1991}. Those are not covered in our setting.


In what follows, the player chooses the next game to play according to a
finite Markov chain, and each game is a random walk in a random
environment, extending the work of \citet{Harmer/Abbott:1999} and
\citet{Pyke:2003}.

In Section \ref{sec:RSRWRE}, the model, which is basically a regime switching random walk in random environments, is described and one of the main characterization results is stated in Theorem \ref{thm:main}, namely that the asymptotic behavior of these models does not always reflect that of independent simple random walks in random environments. Without regime switching, as shown in \citet{Alili:1999}, they are transient, meaning that they  converge to either of $\pm \infty$ with probability $1$, for almost every environment, or they are recurrent, meaning that the limsup and liminf converge respectively  to $\pm \infty$ with probability $1$, for almost every environment.

With regime switching, this is not always the case and the full mathematical analysis of this generalization is the main contribution of this paper. Understanding the asymptotic behavior  under regime switching is key to the development of improved tools and methods relevant to gaming strategies and control issues, when dealing with potential applications in various contexts such as financial market uncertainty \citet{Fink/Klimova/Czado/Stoeber:2017}, investment portfolio solvency \citet{Abourashchi/Clacher/Freeman/Hillier/Kemp/Zhang:2016} and hypothesis testing under richer, more dynamic experimental designs \citet{Carter/Hansen/Steigerwald:2016}.

In Section \ref{sec:hitting}, the possible cases are fully characterized in terms of the rank of the transition matrix that governs regime switching and the dimensions of some random subspaces occurring in Oseledec's ergodic theorem. These results generalize the best known particular cases : (i) when the games are chosen independently, the transition matrix has rank $1$ and the results can be recovered as well from standard arguments applied to random walks in random environments \citep{Alili:1999}; (ii) the periodic choices studied by many authors, starting with \citet{Pyke:2003}, where the transition matrices have full rank; and continuing in a series of papers beginning with \citet{Ethier/Lee:2009}, where the Markov chains associated with the matrices are irreducible. Note that  \citet{Ethier/Lee:2009} consider several other structural choices not covered here. The main contribution in the current paper is in the coverage of new cases including instances of reducible Markovian switching regimes, avoided so far in the literature because of their technical intractability, as exemplified in Remark \ref{rem:bolthausen}.

Finally, the proofs of the results can be found in a series of appendices. They are inspired by the results of \citet{Key:1984} who studied random walks in random environments with bounded increments but no regime switching. The results of \citet{Key:1984} were later refined by several authors, notably
\citet{Bolthausen/Goldsheid:2000}, \citet{Keane/Rolles:2002}, and \citet{Bremont:2002, Bremont:2009}.



\section{Regime switching random walks in random environments}\label{sec:RSRWRE}

We first describe the model and then study its
asymptotic behavior.

\subsection{Model}\label{ssec:model}

First, let $(E,\cE,P)$ be a complete probability space with a measure preserving transformation $T$, assumed to be $\cE$-measurable and ergodic, i.e., the $T$-invariant sigma-field is trivial.
Next, for any $\alpha\in \{1,\ldots,
m\}$,  and any  $k \in \mathbb Z$,   $p^{(\alpha)}_k$ are $(0,1)$-valued variables, where $p_k^{(\alpha)}(Te) = p^{(\alpha)}_{k+1}(e)$, $e \in E$. Hence, the processes
$p^{(\alpha)}$ are stationary and ergodic.

Next, for a given $\alpha\in \{1,\ldots, m\}$,  let $X^{(\alpha)}_n$
be the nearest neighbor random walk in a random environment $e\in E$ defined by the process
$p^{(\alpha)}$, i.e., its so-called quenched law is given by
$$
P\left( X^{(\alpha)}_{n} = k+1 |\; X^{(\alpha)}_{n-1} =k ,\cE\right)(e)
= p^{(\alpha)}_k(e), \; k \in \dZ, \; n\ge 1.
$$

These random walks will be the fortunes of the player as she chooses each game.
Her decision process is based on the Markov chain $G$ on
$\{1,\ldots,m\}$, with transition matrix $Q$.  For example, in Game
C, one can choose $m=4$, $p^{(1)}=p^{(2)}$, $p^{(3)}=p^{(4)}$, where
$p^{(1)}$ is the (deterministic) process determined by  Game  A,
$p^{(3)}$ is the stationary ergodic process determined by  Game  B,
and
$Q = \left(\begin{array}{cccc} 0 & 1 & 0 & 0 \\
0 & 0 & 1 & 0 \\ 0 & 0 & 0 & 1\\ 1 & 0 & 0 & 0\end{array}\right)$.
For Game D, $m=2$ and $Q = \left(\begin{array}{cc} 0.5 & 0.5\\
0.5 & 0.5 \end{array}\right)$.
Further note that for Game C', then $Q = \left(\begin{array}{cc} 0 & 1\\ 1 & 0\end{array}\right)$, $p^{(1)}$ is the (deterministic) process determined by  Game  A,
 and $p^{(2)}$ is the stationary ergodic process determined by  Game  B.

Finally, for any $n\ge 1$, the quenched law of the regime-switching (nearest-neighbor) random walk $X_n$ is defined by 
$$P(X_n=k+1|X_{n-1}=k,G_n=\alpha,\cE)(e) = p_k^{(\alpha)}(e),\quad e\in E, \;k\in \dZ, \; \alpha\in\{1,\ldots,m\}.
$$
Under these assumptions, given $e\in E$,
$(G_n,X_n)$ is an homogeneous Markov chain on $\{1,\ldots, m\}\times
\dZ$ with transition matrix
$$
P_{\alpha i, \beta j}(e) = Q_{\alpha
\beta}P^{(\beta)}_{ij}(e), \quad \alpha,\beta \in \{1,\ldots,
m\},\; i,j \in \dZ,
$$
where $P^{(\beta)}_{i,i+1} = p^{(\beta)}_i $ and
$P^{(\beta)}_{i,i-1}= 1-p^{(\beta )}_{i} = q^{(\beta)}_i $.

For a given environment $e$, the oft-cited \citet{Cogburn:1980} called
$X$ a Markov chain in a random environment, the sequence $G_n$
playing, in his case, the role of the ``environment''. To avoid confusion, we will not make use of this terminology here, the choice for the next game $G_n$ - the regime switching mechanism - depending solely on the previous game selection $G_{n-1}$, independently of the environment.

\begin{remark}\label{rem:bolthausen}
Note that our setting is a particular case of a random walk in a random environment on a strip introduced in \citet{Bolthausen/Goldsheid:2000}. This is also the case for the edge-reinforced random walks in \citet{Keane/Rolles:2002} where their method of proof is similar to ours. However, in both cases, their results cannot be applied in general here since they assumed that the resulting Markov chain $(G_n,X_n)$ has only one communication class \citep[Condition C]{Bolthausen/Goldsheid:2000}, i.e., the Markov chain is almost surely irreducible \citep[Remark 2]{Bolthausen/Goldsheid:2000}. One cannot simply separate the non-communicating classes and use their results since we show that for the reducible case presented in Example \ref{ex:counter}, the asymptotic behavior depends on the random environment, a prohibited phenomenon in the results of \citet{Bolthausen/Goldsheid:2000} and  \citet{Keane/Rolles:2002}. Another example of a reducible case is Game C', where there are two closed classes: $C_0 = \{(\alpha,i); \alpha + i \mbox{ is
even}\}$ and $C_1 = \{(\alpha,i); \alpha + i \mbox{ is odd}\}$.  For this game, we also show in Appendix \ref{app:bolthausen} that the main result of \citet{Bolthausen/Goldsheid:2000} does not apply. 
\end{remark}

One is interested in the asymptotic behavior of process $X$ describing the evolution of the player's fortune.
More precisely, one would like to find conditions under which the
so-called Parrondo's paradox holds, i.e., for any starting point $(\alpha,i)\in
\{1,\ldots,m\}\times \dZ$, and almost every environment $e$,  
$$
P^e_{\alpha
i}\left(\disp\lim_{n\to\infty} X_n =+\infty\right)= P\left(\left. \disp\lim_{n\to\infty} X_n =+\infty\right |\cE, G_0=\alpha,X_0=i\right)(e)=1,
$$
while 
$$
P^e_i\left(\disp\lim_{n\to\infty} X_n^{(\alpha)}
= -\infty \right)= 
P\left(\left. \disp\lim_{n\to\infty} X_n^{(\alpha)}
= -\infty\right|\cE, X_0=i\right)(e)=1.
$$ 

For $\alpha\in \{1,\ldots,m\}$, define the (stationary ergodic)
process $\sigma^{(\alpha)}$ by $\sigma^{(\alpha)}_i =
\frac{q^{(\alpha)}_i}{p^{(\alpha)}_i}$, $i\in \dZ$. Then the
asymptotic behavior of $X^{(\alpha)}$ is completely determined by
the expectation of $\log{\sigma^{(\alpha)}_0}$, as  proven in
\citet[Theorem 2.1]{Alili:1999}.

\begin{theorem}\label{thm:alili} Let $\alpha\in \{1,\ldots,m\}$ be
given and suppose that  $u=E\left(\log{\sigma_0^{(\alpha)}}\right)$
is well defined, with values in $[-\infty,+\infty]$. If $u>0$, then
for any $i\in\dZ$,
$$
P^e_{i}\left(\lim_{n\to\infty} X_n^{(\alpha)} = -\infty
\right)=1, \; e \mbox{ a.s.}
$$
If $u<0$, then for any $i\in\dZ$,
$$
P^e_{i}\left( \lim_{n\to\infty} X_n^{(\alpha)}  = +\infty
\right)=1, \; e \mbox{ a.s.}
$$
Finally, if $u=0$, then for any $i\in\dZ$,
$$
P^e_{i}\left( \liminf_{n\to\infty} X_n^{(\alpha)}  = -\infty ,
\limsup_{n\to\infty} X_n^{(\alpha)}  = +\infty \right)=1, \;e
\mbox{ a.s.}
$$
\end{theorem}

In addition to the trivial case where $p^{(\alpha)}$ is constant,
this result also covers the i.i.d. case first treated by  \citet{Solomon:1975}, i.e.,  where for a given
$\alpha$, $p^{(\alpha)}_i$ are i.i.d. Here are other examples of
stationary ergodic sequences.

\begin{example}[Periodic case]\label{exe:periodic}
\citet{Pyke:2003} studied Parrondo's paradox when for each
$\alpha\in \{1,\ldots,m\}$ $p^{(\alpha)}$ is a deterministic sequence of period
$d_\alpha$, i.e., $p^{(\alpha)}_{k+d_\alpha}=p^{(\alpha)}_k$. These
are particular cases of  stationary ergodic sequences for which each
environment $p^{(\alpha)}(\cdot+j)$, $j\in \{1,\ldots,d_\alpha\}$
has equal probability $1/d_\alpha$, so
$$
u = E\left(\log{\sigma^{(\alpha)}_0}\right) =
\frac{1}{d_\alpha} \sum_{j=1}^{d_\alpha} \log{ \sigma^{(\alpha)}_j
}, \qquad \alpha\in \{1,\ldots, m\}.
$$
This example contains Games A and B as particular cases.
\end{example}

\begin{example}\label{ex:contfraction}
Set $E = (0,1)$ and define $T(e) =
\frac{1}{e}-\lfloor \frac{1}{e}\rfloor$, $e\in
(0,1)$. Then $T$ is a measure-preserving map for the law with
density $f(x) = \frac{1}{\log{2}}\frac{1}{(1+x)}$, $x\in (0,1)$. Set
$p_i(e) = T^i e$ and use the two-sided extension theorem,
e.g., \citet[Theorem 7.1.2]{Durrett:2010}, to obtain a stationary
sequence defined for any $i\in \mathbb{Z}$. It is easy to check that
 $E(\log{\sigma_i}) =
\frac{\log{2}}{2}>0$. Therefore, according to Theorem
\ref{thm:alili}, the associated random walk converges to $-\infty$
with probability one, for almost every environment. Note that $T$ is not invertible.
\end{example}

%
%
%
%
%
%
%
%
%
%
%
%
%
%
%

In the next section, we study the asymptotic behavior of the
process $X$. To reduce the notations,  the random environment is fixed, unless otherwise
specified.

\subsection{Asymptotic behavior}\label{ssec:limit}

For any $\ell \in \dZ$, set $\tau_\ell = \inf\{n\ge 1;
X_n=\ell\}$.  The proof of the following lemma in given in Appendix
\ref{pf:mainlemma}.

\begin{lemma}\label{lem:main}
Let $\ell \in \dZ$  be given. Then $P^e_{\alpha i}\{X_n \le \ell \; i.o. \}=0$
for every $(\alpha,i) \in \{1,\ldots,m\}\times \dZ $ if and only if
for every $(\alpha,i) \in \{1,\ldots,m\}\times \dZ$,
\begin{equation}\label{eq:condlem}
P^e_{\alpha i}(\tau_\ell\ <\infty)   =1 , \mbox{ if  } \; i <
\ell,\mbox{ and } P^e_{\alpha i}(\tau_\ell\ <\infty) <1 , \mbox{
if } \; i> \ell.
\end{equation}
\end{lemma}

The next result is proven in Appendix \ref{app:pf_propexit}.
\begin{proposition}\label{prop:exit} Let $\ell>0$ be given. Then for
any $\alpha\in \{1,\ldots,m\}$,
\begin{equation}\label{eq:sum}
P^e_{\alpha i}(\tau_\ell < \infty \cup  \tau_{-\ell} < \infty ) =1, \quad |i|<\ell.
\end{equation}

\end{proposition}

Conditioning on the first play of the game yields the following.

\begin{proposition}\label{prop:time}
 Let $\ell$ be given and set $\left(f_{i\ell}(e)\right)_\alpha = P^e_{\alpha i}(\tau_\ell <\infty)$, $(\alpha ,i) \in \{1,\ldots,m\}\times \dZ$.
Then for any $\alpha \in \{1,\ldots,m\}$,
\begin{eqnarray*}
\left(f_{i\ell}\right)_\alpha &=& \sum_{\beta=1}^m Q_{\alpha
\beta}\left\{p^{(\beta )}_{i}
\left(f_{i+1,\ell}\right)_\beta+q^{(\beta
)}_{i}\left(f_{i-1,\ell}\right)_\beta\right\}, \quad i\not\in\{
\ell-1,\ell+1\},\\
\left(f_{i\ell}\right)_\alpha &=& \sum_{\beta=1}^m Q_{\alpha
\beta}\left\{p^{(\beta )}_{i}
+q^{(\beta )}_{i}\left(f_{i-1,\ell}\right)_\beta\right\}, \quad i = \ell-1,\\
\left(f_{i\ell}\right)_\alpha &=& \sum_{\beta=1}^m Q_{\alpha
\beta}\left\{p^{(\beta )}_{i}
\left(f_{i+1,\ell}\right)_\beta+q^{(\beta )}_{i}\right\}, \quad i
= \ell+1.
\end{eqnarray*}

\end{proposition}

Set $\|f_{i\ell}(e)\|=\max_{\alpha \in \{1,\ldots,m\} }\left(f_{i\ell}(e)\right)_\alpha$. The following proposition is an interesting consequence of the
previous result. Its proof is given in Appendix \ref{pf:proptime3}.

\begin{proposition}\label{prop:time3}
 Let $\ell$ and $e$ be given.
 If $\|f_{i\ell}(e)\| =
 1$ for some $i<\ell$ then $\|f_{i\ell}(e)\|=1$ for any $i<\ell$. Similarly,  if $\|f_{i\ell}(e)\|=  1$ for some $i>\ell$ then
$\|f_{i\ell}(e)\|=1$ for any $i>\ell$.

Next, suppose further that the Markov chain $G_n$ is irreducible. If $P^e_{\alpha i}(\tau_\ell<\infty) = 1, \; e \mbox{ a.s.}$ for some
$\alpha\in\{1,\ldots,m\}$, and some $i<\ell$, then $P^e_{\alpha
i}(\tau_\ell<\infty) = 1, \; e \mbox{ a.s.}$ for every $\alpha\in
 \{1,\ldots,m\}$ and  any $i<\ell$. Moreover, if
$P^e_{\alpha i}(\tau_\ell<\infty) = 1, \; e \mbox{ a.s.}$ for some
$\alpha\in\{1,\ldots,m\}$, and some $i>\ell$, then $P^e_{\alpha
i}(\tau_\ell<\infty) = 1, \; e \mbox{ a.s.}$ for every $\alpha\in
 \{1,\ldots,m\}$ and  any $i>\ell$.
\end{proposition}

The proof of the following subadditive ergodic theorem for the sequence
$\|f_{\ell+k,\ell+k-1}\|$, $k\ge 1$  is given in Appendix \ref{pf:proptime4}.

\begin{proposition}\label{prop:time4}
For any $n\ge 1$ and any $\ell\in \mathbb{Z}$,
\begin{equation}\label{eq:prodergo}
\|f_{\ell-n,\ell}\| \le \prod_{k=1}^n
\|f_{\ell-k,\ell-k+1}\|.
\end{equation}
Moreover $\|f_{\ell-k,\ell-k+1}\|$, $k\ge 1$, is a stationary
ergodic sequence and
$$
\gamma_-=\lim_{n\to\infty}
\frac{1}{n}\log{\|f_{\ell-n,\ell}(e)\|} \le
E\left\{\log{\|f_{0,1}\|}\right\}\le 0, \; e \mbox{ a.s.}
$$
Similarly,
\begin{equation}\label{eq:prodergo2} \|f_{\ell+n,\ell}\|
\le \prod_{k=1}^n \|f_{\ell+k,\ell+k-1}\|.
\end{equation}
Moreover $\|f_{\ell+k,\ell+k-1}\|$, $k\ge 1$, is a stationary
ergodic sequence and
$$
\gamma_+ = \lim_{n\to\infty}
\frac{1}{n}\log{\|f_{\ell+n,\ell}(e)\|} \le
E\left\{\log{\|f_{1,0}\|}\right\}\le 0, \; e \mbox{ a.s.}
$$
\end{proposition}

We now state  a useful technical result, whose proof is given in Appendix \ref{app:pf-lemma2}.

\begin{lemma}\label{lem:main0} The following statements hold:
\begin{itemize}
\item[(i)]
$P^e_{\alpha i}\left(\disp\lim_{n\to\infty} X_n = +\infty
\right) =1$ for every $(\alpha,i)\in \{1,\ldots,m\}\times \dZ$ if and
only if for every $(\alpha,i,\ell)\in \{1,\ldots,m\}\times \dZ^2$,
\begin{eqnarray}\label{eq:condthm+1} P^e_{\alpha i}(\tau_\ell <\infty)
& = & 1 , \quad i <
\ell,\\
 P^e_{\alpha i}(\tau_\ell <\infty)&<& 1, \quad  i> \ell.\label{eq:condthm+2}
\end{eqnarray}

\item[(ii)] $P^e_{\alpha i}\left(\disp\lim_{n\to\infty} X_n = -\infty  \right) =1$ for every
$(\alpha,i)\in \{1,\ldots,m\}\times \dZ$ if and only if for every
$(\alpha,i,\ell)\in \{1,\ldots,m\}\times \dZ^2$,
\begin{eqnarray}\label{eq:condthm-1} P^e_{\alpha i}(\tau_\ell <\infty)
& = & 1 , \quad  i >
\ell,\\
 P^e_{\alpha i}(\tau_\ell <\infty) &<& 1, \quad i < \ell.  \label{eq:condthem-2}
\end{eqnarray}

\item[(iii)] Suppose further that the Markov chain $G_n$ is irreducible.
\item[] $P^e_{\alpha i}\left(\disp\liminf_{n\to\infty} X_n = -\infty, \limsup_{n\to\infty} X_n = +\infty \right) =1$ for every
$(\alpha,i)\in \{1,\ldots,m\}\times \dZ$, if and only if for every
$(\alpha,\ell)\in \{1,\ldots,m\}\times \dZ$,
\begin{eqnarray}\label{eq:condthm+-} P^e_{\alpha \ell}(\tau_\ell <\infty)
& = & 1.
\end{eqnarray}
\end{itemize}
\end{lemma}

We  are now in a position to state the first main result,
proven in Appendix \ref{pf:mainthm}.

\begin{theorem}\label{thm:main}
Set $ \gamma_\pm =\lim_{n\to\infty} \frac{1}{n}\log\|f_{\pm
n,0}\|$. Then under no additionnal condition, there holds
$\max(\gamma_+ ,\gamma_-)=0$. If in addition $(G_n,X_n)$ is irreducible, then one  of
the following three mutually exclusive cases occur:
\begin{enumerate}
\item If
$\gamma_+ < 0$ and $\gamma_-= 0$, then for every $(\alpha,i)\in \{1,\ldots,m\}\times \dZ$,
$$
P^e_{\alpha i}\left(\disp\lim_{n\to\infty} X_n = +\infty
\right) =1, \; e  \mbox{ a.s.}
$$

\item If
$\gamma_+ = 0$ and $\gamma_- <0$, then for every $(\alpha,i)\in \{1,\ldots,m\}\times \dZ$,
$$
P^e_{\alpha i}\left(\disp\lim_{n\to\infty} X_n = -\infty
\right) =1,  \; e \mbox{ a.s.}
$$

\item If
$\gamma_+ = 0 =\gamma_-$, then for every $(\alpha,i)\in \{1,\ldots,m\}\times \dZ$,
$$
P^e_{\alpha i}\left(\disp\liminf_{n\to\infty} X_n = -\infty,
\limsup_{n\to\infty} X_n = +\infty  \right) =1,  \; e  \mbox{ a.s.}
$$
\end{enumerate}
\end{theorem}

This result shows that, under the hypothesis of irreducibility of the chain $(G_n,X_n)$ - the assumption
common to all previous instances in the literature quoted thus far in the present paper - the behavior of regime switching random walks in random environments mimics the (recurrence, transience-to-the-left, transience-to-the-right) trichotomy exhibited by random walks not subjected to regime switching.

In the next section we remove the hypothesis of irreducibility of the chain $(G_n,X_n)$ and obtain a new, mutually exclusive breakdown of the asymptotic behavior, especially of interest in the more difficult third case of Theorem \ref{thm:main}.

\section{Other criteria for transience and recurrence}\label{sec:hitting}

First, we express the relationship between hitting probabilities.
These will be needed for computation purposes. Then it will be shown that  $\gamma_+$ and $\gamma_-$ are related to dimensions of some random spaces through the famous Oseledec's Theorem stated in Appendix \ref{app:oseledec}.

\subsection{Recursive formulas for hitting
probabilities}\label{ssec:recurs}

For any given $\ell\in\dZ$ recall that $\tau_\ell = \inf\{n\ge 1;
X_n=\ell\}$ and set $\tilde \tau_\ell = \inf\{n\ge 0;
X_n=\ell\}$. For any choice of $\alpha,\beta\in\{1,\ldots,m\}$ set
$\left(\tilde f^{(\beta)}_{i\ell}\right)_\alpha (e)
=P^e_{\alpha i}\left(\tilde \tau_\ell<\infty,G_{\tilde
\tau_\ell}=\beta\right)$ and similarly $\left( f^{(\beta)}_{i\ell}\right)_\alpha (e)
=P^e_{\alpha
i}\left(\tau_\ell<\infty,G_{\tau_\ell}=\beta \right)$.
It is easy to check that $\left(\tilde f^{(\beta)}_{\ell\ell}\right)_\alpha  =
\delta_{\alpha\beta}$, for any $\alpha,\beta\in\{1,\ldots,m\}$.
Also, when $i\neq \ell$, then $\left( f^{(\beta)}_{i\ell}\right)_\alpha  =
\left(\tilde f^{(\beta)}_{i\ell}\right)_\alpha $. It then follows that
\begin{equation}\label{eq:hittinggen}
\left(\tilde f^{(\beta)}_{i\ell}\right)_\alpha = \sum_{k=1}^m Q_{\alpha k}
\left\{p^{(k)}_{i} \left(\tilde f^{(\beta)}_{i+1,\ell}\right)_k +
q^{(k)}_{i} \left(\tilde f^{(\beta)}_{i-1,\ell}\right)_k\right\}.
\end{equation}

First, let $\Delta_i$ be the random
diagonal matrix with entries $\left\{p^{(1)}_{i}, \ldots,
p^{(m)}_{i}\right\}$.

Then, using \eqref{eq:hittinggen}, $\tilde f^{(\beta)}_{\ell\ell} = e^{(\beta)}$,
where $e^{(\beta)}_\alpha = \delta_{\alpha\beta}$, $\alpha\in
\{1,\ldots,m\}$, and
\begin{equation}\label{eq:main1}
\tilde f^{(\beta)}_{i\ell} = M_i \tilde f^{(\beta)}_{i+1,\ell}+
N_i \tilde f^{(\beta)}_{i-1,\ell}, \quad i\neq \ell, \beta\in
\{1,\ldots,m\},
\end{equation}
where
\begin{equation}\label{eq:MN}
M_i = Q\Delta_i , \qquad N_i = Q(I-\Delta_i),
\end{equation}
so that $M_i+N_i=Q$  for any  $i\in \dZ$. Note that for $i\neq
\ell$, $P^e_{\alpha i}(\tau_\ell<\infty) = 1$ for all
$\alpha\in\{1,\ldots,m\}$ if and only if
$
f_{i\ell} (e) = \tilde f_{i\ell}(e) = \sum_{\beta=1}^m
\tilde f^{(\beta)}_{i\ell}(e) = \mathbf{1}$, with $\mathbf{1}_\alpha=1$ for all $\alpha\in\{1,\ldots,m\}$.
Further note  that $\tilde f_{\ell\ell}=\mathbf{1}$.


\subsection{First case: $Q$ has rank 1}\label{ssec:caseind}

It then follows that for some positive vector $\pi$, $Q_{\alpha\beta}=\pi_\beta >0$ for all $\alpha,\beta
\in\{1,\ldots,m\}$. Note that in this case, the chain $(G_n,X_n)$ is obviously irreducible. This corresponds to choosing the regimes
independently, so $X$ is a Markov chain with transition matrix $P=P(e)$
given by $P_{ij} = p_i$ if $j=i+1$, $P_{ij} =
q_i$ if $j=i-1$, and $P_{ij}=0$, whenever $|j-i|>1$, where $p_i
= \sum_{\beta=1}^m \pi_\beta p^{(\beta)}_i$, $q_i = 1-p_i$,
$i\in\dZ$. Then it follows from \eqref{eq:main1} that $
f^{(\beta)}_{i\ell} = g_{i\ell}^{(\beta)}\mathbf{1}$, for every
$i\neq \ell$. Set $\sigma_i = \frac{q_i}{p_i}$, and $g_{i\ell} =
\sum_{\beta=1}^m g_{i\ell}^{(\beta)}$,  $i\in \dZ$.
Next, it is easy to check that for every $\beta\in\{1,\ldots,m\}$,
\begin{equation}\label{eq:betaind}
P^e_{\alpha i}\left(G_{\tau_\ell} = \beta| \tau_\ell <\infty \right)
= \frac{g_{i\ell}^{(\beta)}(e) }{ g_{i\ell}(e) } =  \left\{\begin{array}{ll}
\pi_\beta\frac{
q^{(\beta)}_{\ell+1}(e)}{q_{\ell+1}(e)} & \mbox{ for } i>\ell,\\
& \\
 \pi_\beta\frac{ p^{(\beta)}_{\ell{\color{red} -}1}(e)}{p_{\ell{\color{red} -}1}(e)} &  \mbox{ for } i<\ell.
 \end{array}\right.
 \end{equation}
Finally, since $X$ is itself a random walk in a random environment,
one can apply Theorem \ref{thm:alili}, to obtain the following
corollary.

\begin{corollary}\label{cor:ind}
If $E(\log{\sigma_0})>0$, then for any $(\alpha,i)
\in\{1,\ldots,m\}\times \dZ$,
$$
P^e_{\alpha i}\left( \disp \lim_{n\to\infty} X_n = -\infty
\right)=1, \; e \mbox{ a.s.}
$$
If $E(\log{\sigma_0})<0$, then for any $(\alpha,i)
\in\{1,\ldots,m\}\times \dZ$,
$$
P^e_{\alpha i}\left( \disp \lim_{n\to\infty} X_n = +\infty
\right)=1, \; e  \mbox{ a.s.}
$$
 If
$E(\log{\sigma_0})=0$, then for any $(\alpha,i)
\in\{1,\ldots,m\}\times \dZ$,
$$
P^e_{\alpha i}\left(  \disp  \liminf_{n\to\infty} X_n =
-\infty, \limsup_{n\to\infty} X_n = +\infty \right)=1, \; e
\mbox{ a.s.}
$$
\end{corollary}

\begin{example}

For Game D, one finds that  $p$ is 3-periodic with values
$\{.299,.624,.624\}$. Using Corollary \ref{cor:ind}, one obtains
that
$
E(\log{\sigma_0}) = \frac{1}{3}\log{\mu}$,
where $\mu$ given by formula \eqref{eq:mu}. In this specific
example, $\mu = 0.8512 <1$, so that for any $(\alpha,i)\in
\{1,2\}\times \dZ$, $P^e_{\alpha i}\left(\disp \lim_{n\to\infty}
X_n = +\infty\right)=1$ a.s. More generally, the value of $\mu$ depends on $\pi_1$ which is the probability of choosing Game A. It then follows that 
$$
\mu(\pi_1) = \left(\frac{1}{.099+.4\pi_1}-1\right)\left(\frac{1}{.749-.25\pi_1}-1\right)^2.
$$
Figure \ref{fig:graphmuD} illustrates
that $\mu(\pi_1)<1$ quite often.

\begin{figure}[h!]
\begin{center}
\includegraphics[scale = 0.35]{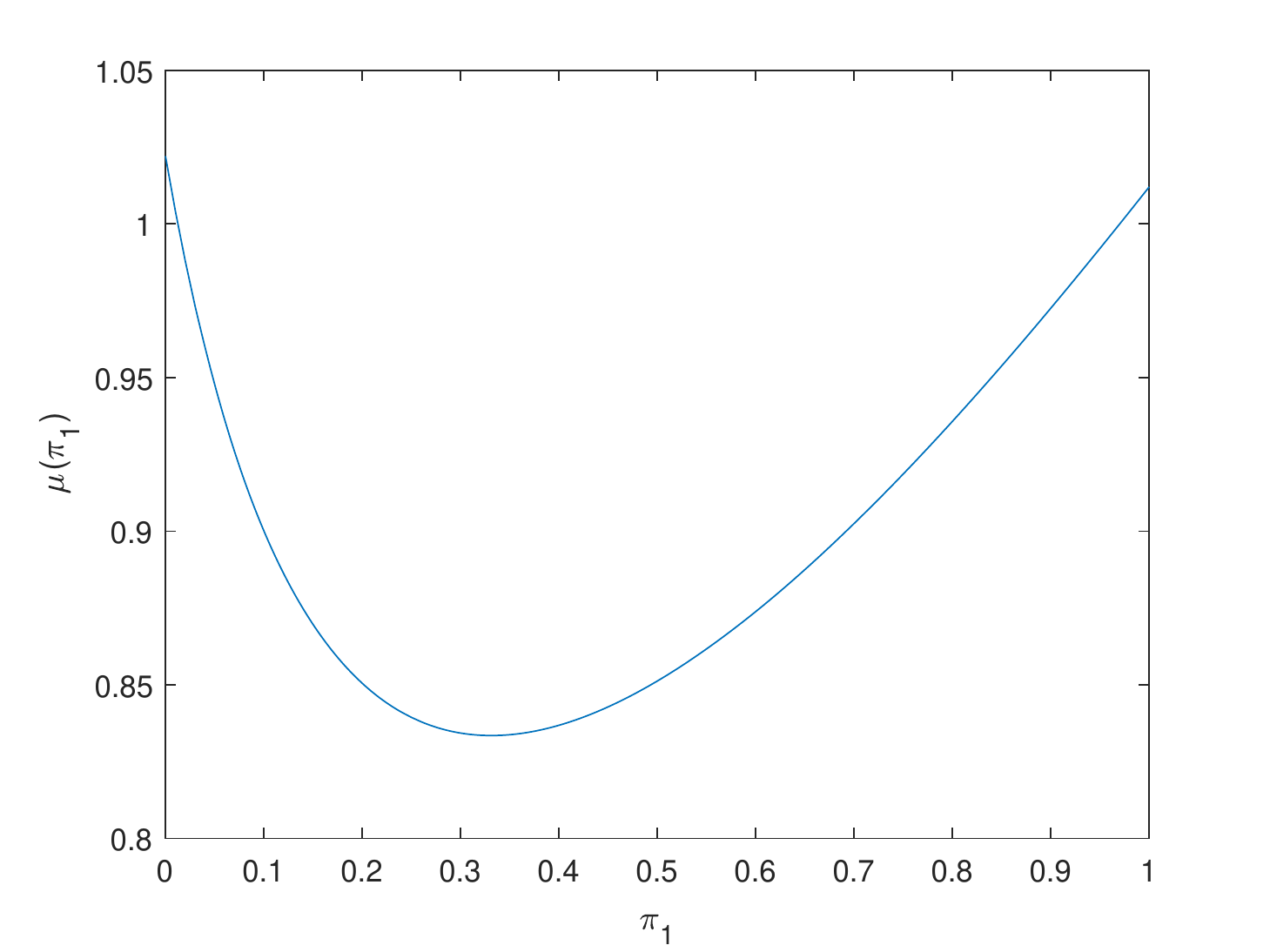}
\caption{Graph of $\mu$ for Game D as a function of the probability $\pi_1$ of choosing at
Game  A. }\label{fig:graphmuD}
\end{center}
\end{figure}
\end{example}
\bigskip

\subsection{Second case: $Q$ has full rank}\label{ssec:full}

Since  $Q$ is invertible, it follows that  $M_i$ and $N_i$, as
defined in \eqref{eq:MN}, are invertible.
For $i \in \dZ$, set  $A_i =  \left(\begin{array}{cc} M_i^{-1} & -\sigma_i\\
I & 0 \end{array}\right) $,  where $\sigma_i = M_i^{-1} N_i =
\Delta_i^{-1}-I$.
 Since each $A_i$ is invertible and the
sequence of $2m\times 2m$ matrices $A_i$ is stationary and ergodic,
it follows from Oseledec's Theorem (Theorem \ref{thm:oseledec}) that
with probability 1, the random sets
$$
\bar V_0 =  \left\{v\in \mathbb{R}^{2m}: \lim_{n\to\infty}
n^{-1}\log{\|A_n \cdots A_1 v\|} \le 0 \right\}
$$
and
$$
\bar V_{0-} =  \left\{v\in \mathbb{R}^{2m}: \lim_{n\to\infty}
n^{-1}\log{ \|A_n \cdots A_1 v\|} < 0\right\}
$$
are subspaces with deterministic dimensions $\bar d_0$ and $\bar
d_{0-}$ respectively, provided $\log^+ {\|A_1\|}$ is integrable.
Note that for any $i\in\dZ$, $A_i \un = \un$.
The norm ${\| \|}$ is arbitrary but fixed throughout since they are all equivalent.

Also, with probability 1, the random sets
$$
\tilde V_0 =  \left\{v\in \mathbb{R}^{2m}: \lim_{n\to\infty}
n^{-1}\log{\left\|\left(A_n \cdots A_1\right)^{-1} v\right\|} \le 0
\right\}
$$
and
$$
\tilde V_{0-} =  \left\{v\in \mathbb{R}^{2m}: \lim_{n\to\infty}
n^{-1}\log{ \left\|\left(A_n \cdots A_1\right)^{-1} v\right\|}
<0\right\}
$$
are subspaces with deterministic dimensions $\tilde d_0$ and $\tilde
d_{0-}$ respectively, provided $\log^+ {\left\|A_1^{-1} \right\|}$
is integrable.\\

The proof of the next result is given in Appendix
\ref{pf:thmoseledec}.

\begin{theorem}\label{thm:oseledecfull}
Suppose that $Q$ is invertible and assume that $\log^+ {\|A_1\|}$
and $\log^+ {\left\|A_1^{-1} \right\|}$ are integrable.  Then $\bar d_0\ge m$, $\tilde d_0\ge m$, $\bar d_0 > \bar d_{0-}$, and  $\tilde d_0 > \tilde d_{0-}$. Furthermore,
\begin{enumerate}
\item If $\bar d_0=m$, then $\gamma_+=0$ and  $P^e_{\alpha i}(\tau_\ell <\infty )=1$, $e$
a.s., for any $\alpha\in\{1,\ldots,m\}$ and any $i>\ell$.

\item $\bar d_{0-} \ge  m$ if and only if $\gamma_+ <0$. In this case, there holds
$\lim_{i\to \infty} P^e_{\alpha i}(\tau_\ell <\infty )=0$ for any
$\alpha\in\{1,\ldots,m\}$.

\item If $\tilde d_0=m$, then  $\gamma_-=0$,  and  $P^e_{\alpha i}(\tau_\ell <\infty )=1$, $e $
a.s., for any $\alpha\in\{1,\ldots,m\}$ and any $i <\ell$.

\item $\tilde d_{0-} \ge  m$ if and only if $\gamma_- <0$. In this case, there holds
$\lim_{i\to -\infty}P^e_{\alpha i}(\tau_\ell <\infty )=0$ for any
$\alpha\in\{1,\ldots,m\}$.
\end{enumerate}
\end{theorem}
\begin{remark} \label{rem:ruelle0}
It follows from Theorem \ref{thm:main} that  $\max(\gamma_+,\gamma_-)=0$, so Theorem \ref{thm:oseledecfull} implies that one cannot have at the same time $\bar d_{0-} \ge  m$ and
 $\tilde d_{0-}\ge m$.
Hence, if  $\bar d_{0-} \ge  m$, there holds $\gamma_+<0$ and $\gamma_- = 0$. Combining
Theorem \ref{thm:oseledecfull} with Proposition \ref{prop:time3} first yields $\|f_{i\ell}(e)\|<1$ for any $i>\ell$.
Proposition \ref{prop:time4} together with Proposition \ref{prop:time3} yield that
$\|f_{i\ell}(e)\|=1$ holds for any $i<\ell$. If in addition, $\tilde d_0 = m$, then
$$
P^e_{\alpha i}\left(\disp\lim_{n\to\infty} X_n = +\infty
\right) =1, \; e  \mbox{ a.s.}
$$
for every $(\alpha,i)\in \{1,\ldots,m\}\times \dZ$, using Lemma \ref{lem:main0}. \\

On the other hand, if $\tilde d_{0-} \ge  m$,  there holds $\gamma_-<0$, $\gamma_+=0$. If in addition, $\bar d_0 = m$, then  it follows in a similar fashion that
$$
P^e_{\alpha i}\left(\disp\lim_{n\to\infty} X_n = -\infty
\right) =1, \; e  \mbox{ a.s.}
$$
for every $(\alpha,i)\in \{1,\ldots,m\}\times \dZ$.\\

Note that these two transient behavior occur without the assumption that $(G_n,X_n)$ is irreducible.
Finally, if $\bar d_{0-} < m$ and $\tilde d_{0-}  < m $, then $\gamma_+=\gamma_-=0$; nevertheless,  one cannot show in general that a form of recurrence occurs, i.e.,
$$
P^e_{\alpha i}\left(\disp\liminf_{n\to\infty} X_n = -\infty,
\limsup_{n\to\infty} X_n = +\infty  \right) =1,  \; e  \mbox{ a.s.}
$$
for every $(\alpha,i)\in \{1,\ldots,m\}\times \dZ$.
 In fact, a counterexample is given in Example \ref{ex:counter}, where $(G_n,X_n)$ is not irreducible.
\end{remark}

Note that to obtain a ``Parrondo's paradox'' which is basically a transient phenomenon, we do not need the hypothesis of irreducibility. In fact, combining Theorem \ref{thm:alili} and Remark \ref{rem:ruelle0},  we end up with the following sufficient condition.

\begin{corollary}\label{cor:parrondo}
If for every $\alpha \in \{1,\ldots,m\}$, there holds $ E\left(\log{\sigma_0^{(\alpha)}}\right) > 0 $,    $\tilde d_0 = m$, and $\bar d_{0-}\ge m$, then  for any
$(\alpha,i)\in \{1,\ldots,m\}\times \dZ$ and any $e$ a.s.,
$$
P^e_i\left(\disp \lim_{n\to\infty} X_n^{(\alpha)} = -\infty\right) =1  \quad \text{ and } \quad
P^e_{\alpha
i}\left(\disp\lim_{n\to\infty} X_n = +\infty\right) =1.
$$
\end{corollary}

Note also that for an invertible measure preserving mapping $T$, it follows from
\citet[Theorem 3.1]{Ruelle:1979}  that the Lyapunov exponents
$\tilde \lambda_1, \ldots, \tilde \lambda_{2m}$ associated with the time reversed sequence
$
\left(A_n \cdots A_1\right)^{-1}
$ in Oseledec's Theorem (Theorem \ref{thm:oseledec})
are $-\bar\lambda_{2m}, \ldots, -\bar\lambda_1$, that is, minus those associated with
$\left(A_n \cdots A_1\right)$. Hence,
$\tilde d_{0-} = \dim{\tilde V_{0-}} = 2m - \bar d_0$ and $\tilde
d_0 = \dim{\tilde V_0} = 2m - \bar d_{0-}$.
 In particular, if $\bar d_0=m$, then $\tilde d_{0-}=m$
and if $\bar d_{0-}=m$, then  $\tilde d_0=m$.

The proof of the following corollary follows directly from Theorem \ref{thm:oseledecfull} and Remark \ref{rem:ruelle0}.

\begin{corollary}\label{cor:ruelle}
Suppose that the measure preserving mapping $T$ defined in Section \ref{sec:RSRWRE} is invertible.
If $\bar d_0 = m$, then
 for every $(\alpha,i)\in
\{1,\ldots,m\}\times \dZ$,
$P^e_{\alpha i}\left(\disp\lim_{n\to\infty}
X_n = -\infty\right) =1$, $e$ a.s.
Similarly, if $\tilde d_{0} =m$,  then for any
$(\alpha,i)\in \{1,\ldots,m\}\times \dZ$, $P^e_{\alpha
i}\left(\disp\lim_{n\to\infty} X_n = +\infty\right) =1$,  $e$ a.s..
\end{corollary}

\begin{example}\label{ex:periodic2}[Periodic probabilities] Suppose that the processes $p^{(\alpha)}$ are periodic for all $\alpha\in \{1,\ldots,m\}$, and denote by $\mathfrak{p}$ the least common multiple of their respective periods. It then follows that the measure preserving mapping is invertible. In fact, if $e$ is such that $e^{(\alpha)}(i) = p_i^{(\alpha)}$, then $E = \{e, Te, \ldots, T^{\mathfrak{p}-1}e\}$, where $Te(i) = e(i+1)$, $i\in \dZ$. Further set  $\cA = A_\mathfrak{p} \cdots A_1$ .
Then $\bar d_0$ is the number of eigenvalues of $\cA$ less than or
equal to $1$ in absolute value, while $\tilde d_0$ is the number of
eigenvalues of $\cA$ greater of equal to $1$ in absolute value.
Game C is an
example of a $Q$ with full rank and periodic probabilities with period $\mathfrak{p}=3$. In this
case
$m=4$, $\bar d_0=6$, and $\tilde d_0=4$.
 As a result, from Corollary \ref{cor:ruelle}, $P^e_{\alpha
i}\left(\displaystyle\lim_{n\to\infty} X_n = +\infty\right) =1$,  $e$ a.s., for every $(\alpha,i)\in \{1,\ldots,m\}\times \dZ$. This explains why Game C has a paradoxical behavior.

 Suppose now that one alternates between Games A and B, i.e., one plays Game C'. Then $Q = \left(\begin{array}{cc} 0 & 1\\ 1 & 0\end{array}\right)$,
$m=2$, $\bar d_0=2$, and $\tilde d_0=4$.  

From Corollary \ref{cor:ruelle}, $P^e_{\alpha
i}\left(\displaystyle\lim_{n\to\infty} X_n = -\infty\right) =1$,  $e$ a.s., for every $(\alpha,i)\in \{1,\ldots,m\}\times \dZ$. Hence, in this case, the game does not have a paradoxical behavior.

\end{example}

\begin{example}\label{ex:counter}
Suppose that  $Q = \left(\begin{array}{cc} 0 & 1\\ 1 & 0\end{array}\right)$, and consider the case of periodic probabilities of period $\mathfrak{p}=2$.  Then, starting from $(G_0,X_0)= (1,0)$ and environment $e$, $(X_n)$ is a random walk, and if  $\rho_k(e)  = \frac{P^e(X_{n+1}=k-1|X_n=k)}{P^e(X_{n+1}=k+1|X_n=k)}$, then $\rho_k(e)= \sigma_1^{(2)}(e) = \frac{q_1^{(2)}(e)}{p_1^{(2)}(e)}$ if $k$ is odd, and $\rho_k(e) = \sigma_0^{(1)}(e) =\frac{q_0^{(1)}(e)}{p_0^{(1)}(e)}$ if $k$ is even. Therefore, for any $k\in \dZ$,  $\rho_k(e)\rho_{k+1}(e) =  \lambda_1(e) = \sigma_2^{(1)}(e)\sigma_1^{(2)}(e)$, and
$\lambda_2(e) = \lambda_1(Te) = \sigma_1^{(1)}(e)\sigma_2^{(2)}(e)$. Note that
the eigenvalues of $A_2(e)A_1(e)$ are $1,1,\lambda_1(e), \lambda_2(e)$, which are the same as the eigenvalues of $A_2(Te)A_1(Te)$.
As in \citet{Alili:1999}, set
$$
S(e) = \sum_{n=0}^\infty \rho_1(e) \cdots \rho_n(e), \quad
F(e) = \sum_{n=0}^\infty \frac{1}{\rho_{-1}(e)} \cdots \frac{1}{\rho_{-n}(e)}.
$$
Then $S$ is finite iff $\lambda_1<1$ and $F$ is finite iff $\lambda_1 >1$.

As a numerical example, for any  $i\in \mathbb{Z}$, define $p_i^{(1)}= 0.49$, $p_{2i}^{(2)} =  0.48$, and $p_{2i-1}^{(2)} = 1/1.95$.
Here $E = \{e,Te\}$, where $e(i) = p_i^{(2)}$, so $Te (i) = p_{i+1}^{(2)}$.
Then, if $k$ is even, $\rho_k(e) = 1.0408$, while if $k$ is odd, then $\rho_k(e) = 0.95$  and $\rho_k(Te) = 1.0833$. Hence,
$\lambda_1(e) = \rho_0(e)\rho_1(e) = 0.9888 <1$ and $\lambda_2(e) = \rho_0(Te)\rho_1(Te) = 1.1276>1$.    As a result, $\bar d_0=3 = \tilde d_0$.
Also, $S(e) <
\infty$, $S(Te)= \infty$, $F(e) = \infty$, and $F(Te)< \infty$. Thus, starting from $P_{(1,0)}^e(X_n \to +\infty)=1$ and $P_{(1,0)}^{Te}(X_n \to -\infty)=1$.

On the other hand, starting from $(G_0,X_0)= (2,0)$ and environment $e$, $(X_n)$ is a random walk with $\rho_k(e) = \sigma_1^{(1)}(e) = \frac{q_1^{(1)}(e)}{p_1^{(1)}(e)}$ if $k$ is odd, and $\rho_k(e) = \sigma_0^{(2)}(e) = \frac{q_0^{(2)}(e)}{p_0^{(2)}(e)}$ if $k$ is even. As a result, if $k$ is even, $\rho_k(e) = 1.0833 $ and $\rho_k(Te) = 0.95$, while if $k$ is odd, then $\rho_k(e) = 1.0408$.  Hence
$\rho_0(e)\rho_1(e) = 1.1276>1$ and $\rho_0(Te)\rho_1(Te) = 0.9888 <1$.
Therefore $S(e) =  \infty$, $S(Te) < \infty$, $F(e) < \infty$, and $F(Te) =  \infty$. Thus, $P_{(2,0)}^e(X_n \to -\infty)=1$ and $P_{(2,0)}^{Te}(X_n \to +\infty)=1$.
Summarizing, for the same environment, the asymptotic behavior of the random walk depends on the starting point,
and for the same starting point, the asymptotic behavior of the random walk depends on the environment.
This shows that the case $\bar d_0>m$ and $\tilde d_0>m$ can lead to chaotic behavior.

\end{example}

\bigskip

\subsection{General case: $Q$ has rank $r$ }\label{ssec:gen}

Suppose that $Q$ has rank $1<r < m$. One can assume, without loss of
generality that $Q = \left(\begin{array}{c} \pi
\\ \Theta \pi\end{array}\right)$, where $\pi \in
\dR^{r\times m}$ has rank $r$, and $\Theta \in \dR^{(m-r)\times r}$, so
that $\Theta \mathbf{1}_r = \mathbf{1}_{(m-r)}$.
 Further let $\Delta_i^{(1)}$ be the $r\times r $ diagonal matrix formed with the first $r$ rows and columns of $\Delta_i$, and let $\Delta_i^{(2)}$ stand for the diagonal matrix formed with the last $m-r$ rows and columns of $\Delta_i$. Finally, let $\pi^{(1)}$ be the matrix composed for the first $r$ columns of $\pi$ and set $\pi^{(2)}$  for the remaining $m-r$ columns of $\pi$.

For  $i\in \dZ$, set $\check M_ i = \pi^{(1)}\Delta_i^{(1)} + \pi^{(2)}\Delta_i^{(2)}\Theta$ and $\check N_ i = \pi^{(1)}\left(I-\Delta_i^{(1)} \right)+ \pi^{(2)}\left(I-\Delta_i^{(2)}\right)\Theta$.
%

\begin{hyp}\label{hyp:inv}
With probability $1$, $\check M_i$ and  $\check N_i$ are  invertible.
\end{hyp}
Under this assumption, for any $i\in \dZ$, define the matrices 
$\check A_i = \left(\begin{array}{cc} {\check M_i}^{-1} & - \check \sigma_i \\ \mathbf{1}_r & 0\end{array}\right)$ and
$\check B_i =  \left(\begin{array}{cc} \check N_i^{-1} & -\check \sigma_i^{-1}\\
I & 0 \end{array}\right) $,  where $\check \sigma_i = {\check M_i}^{-1}  \check N_i$.
Note that both $\check M_i$  and  $\check N_i$ are stationary ergodic sequences since $\Theta$ is not random.
As before, it follows from Oseledec's Theorem (Theorem \ref{thm:oseledec}) that
with probability 1, the random sets
$$
\check V_0 =  \left\{v\in \mathbb{R}^{2r}: \lim_{n\to\infty}
n^{-1}\log{\|\check A_n \cdots \check A_1 v\|} \le 0 \right\}
$$
and
$$
\check V_{0-} =  \left\{v\in \mathbb{R}^{2r}: \lim_{n\to\infty}
n^{-1}\log{ \|\check A_n \cdots \check A_1 v\|} < 0\right\}
$$
are subspaces with deterministic dimensions $\check d_0$ and $\check
d_{0-}$ respectively, provided $\log^+ {\|\check A_1\|}$ is integrable.
Also, with probability 1, the random sets
$$
\inw V_0 =  \left\{v\in \mathbb{R}^{2r}: \lim_{n\to\infty}
n^{-1}\log{\left\|\left(\check A_n \cdots \check A_1\right)^{-1} v\right|} \le 0
\right\}
$$
and
$$
\inw V_{0-} =  \left\{v\in \mathbb{R}^{2r}: \lim_{n\to\infty}
n^{-1}\log{ \left\|\left(\check  A_n \cdots \check A_1\right)^{-1} v\right\|}
<0\right\}
$$
are subspaces with deterministic dimensions $\inw d_0$ and $\inw
d_{0-}$ respectively, provided $\log^+ {\left\|\check A_1^{-1} \right\|}$
is integrable.

The proof of the following result is given in Appendix
\ref{pf:oseledecgen}.

\begin{theorem}\label{thm:oseledecgen}
Suppose that Hypothesis \ref{hyp:inv} holds and that $\log^+ {\|\check A_1\|}$
and $\log^+ {\left\|\check A_1^{-1} \right|}$ are integrable. Then $\check d_0\ge r$, $\inw d_0\ge r$, $\check d_0 > \check d_{0-}$, and  $\inw d_0 > \inw d_{0-}$. Furthermore,
\begin{enumerate}
\item If $\check d_0=r$, then $\gamma_+=0$ and  $P^e_{\alpha i}(\tau_\ell <\infty )=1$, $e$
a.s., for any $\alpha\in\{1,\ldots,m\}$ and any $i>\ell$.

\item $\check d_{0-} \ge  r$ if and only if $\gamma_+ <0$. In this case,
$\lim_{i\to \infty} P^e_{\alpha i}(\tau_\ell <\infty )=0$ for any
$\alpha\in\{1,\ldots,m\}$.

\item If $\inw d_0=r$, then  $\gamma_-=0$,  and  $P^e_{\alpha i}(\tau_\ell <\infty )=1$, $e $
a.s., for any $\alpha\in\{1,\ldots,m\}$ and any $i <\ell$.

\item $\inw d_{0-} \ge r$ if and only if $\gamma_- <0$. In this case,
$\lim_{i\to -\infty}P^e_{\alpha i}(\tau_\ell <\infty )=0$ for any
$\alpha\in\{1,\ldots,m\}$.
\end{enumerate}
\end{theorem}

\begin{remark} Theorem \ref{thm:oseledecgen} implies that one cannot have at the same time $\check d_{0-} \ge  r$ and $ \inw d_{0-}\ge r$.
Hence, if $\inw d_{0} = r$  and $\check d_{0-} \ge r$, $\gamma_- = 0 $, $\gamma_+<0$,  and it follows from Lemma \ref{lem:main0}  that
$$
P^e_{\alpha i}\left(\disp\lim_{n\to\infty} X_n = +\infty
\right) =1, \; e  \mbox{ a.s.}
$$
for every $(\alpha,i)\in \{1,\ldots,m\}\times \dZ$.
On the other hand, if $\check d_0=r$ and $\inw d_{0-}\ge r$, then  $\gamma_+ = 0$, $\gamma_-<0$,   and
$$
P^e_{\alpha i}\left(\disp\lim_{n\to\infty} X_n = -\infty
\right) =1, \; e  \mbox{ a.s.}
$$
for every $(\alpha,i)\in \{1,\ldots,m\}\times \dZ$.

\end{remark}

These results are summarized in the following corollary if the mapping $T$ is invertible.

\begin{corollary}\label{cor:ruelle2}
Suppose that the mapping $e \mapsto Te$ is invertible.  If $\check d_0 = r$, then
$P^e_{\alpha i}\left(\disp\lim_{n\to\infty}
X_n = -\infty\right) =1$, $e$ a.s.,  for every $(\alpha,i)\in
\{1,\ldots,m\}\times \dZ$.

Also, if $\check d_{0-} = r$,  then   $P^e_{\alpha
i}\left(\disp\lim_{n\to\infty} X_n = +\infty\right) =1$,  $e$ a.s., for every
$(\alpha,i)\in \{1,\ldots,m\}\times \dZ$.
\end{corollary}

 Using our methodology when $r=1$, one recovers Theorem \ref{thm:alili} due to \citet{Alili:1999}.

\begin{corollary}\label{cor:alili1}
Suppose $r=1$ and set $u= E\left(\log{\sigma_0}\right)$. Then $u<0$ if and only if $\check d_{0-} =1$;
$u>0$   if and only if $\inw d_{0-} =1$; $u=0$ if and only if $\check d_{0-} =\inw d_{0-}=0$.
\end{corollary}
\begin{proof}
Note that is this case, $\check A_n\cdots \check A_1 = \left(\begin{array}{cc}1+U_n & -U_n\\
1+U_{n-1}& -U_{n-1}\end{array}\right)$, where $U_n = \sigma_1+\sigma_1\sigma_2+\ldots+\sigma_1 \cdots \sigma_n$.

Since $A_n \left(\begin{array}{c}1\\
1\end{array}\right) = \left(\begin{array}{c}1\\
1\end{array}\right)$,  $\check A_n\cdots \check A_1 \left(\begin{array}{c}0\\
1\end{array}\right) = \left(\begin{array}{c}U_n\\
U_{n-1}\end{array}\right)$, and \\
$\left(\check A_n\cdots \check A_1\right)^{-1} \left(\begin{array}{c}0\\
1\end{array}\right) = \left(\begin{array}{c}U_n/s_n\\
(1+U_{n})/s_n\end{array}\right)$, where $s_n = \sigma_1 \cdots \sigma_n$.
Next, it is easy to check that $\frac{1}{n}\log{U_n} \to \max(0,u)$ and $\frac{1}{n}\log\left\{(1+U_n)/s_n\right\} \to \max(0,-u)$, as $n\to\infty$.
 It then follows from Oseledec's Theorem (Theorem \ref{thm:oseledec}) that
$u<0$ if and only if $\check d_{0-} =1$, $\check d_0 = 2$, $\inw d_{0-}=0$ and $\inw d_{0}=1$.  Just take $v = (U_\infty, 1+U_\infty)^\top$ to get $\check \lambda_1 = u<0=\lambda_2$. Here $U_\infty = \left(1-e^u\right)^{-1}$. Also take $v = (0 ,1)^\top$ to obtain $\inw \lambda_2 = -u>0 = \inw \lambda_1$.
Similarly, $u>0$ if and only if $\check d_{0-} =0$, $\check d_0 = 1$, $\inw d_{0-}=1$ and $\inw d_{0}=2$. Finally, $u=0$ if and only if $\check d_{0-} =0=\inw d_{0-}$, and $\check d_0 = 2=\inw d_{0}$, since in this case, taking $v = (0,1)^\top$, one gets $\check \lambda_1 = \check \lambda_2 = \inw \lambda_1 = \inw \lambda_2 =0$.
\end{proof}

\begin{example}\label{ex:weird}
 Suppose
 $Q = \frac{1}{24}\left( \begin{array}{ccc}  8 &  8 & 8\\ 6 & 6 & 12\\ 7 & 7 & 10 \end{array}\right)$. Then $\Theta = \left( \frac{1}{2} \quad  \frac{1}{2}\right)$.

In this  example, $
\check M_i = \frac{1}{24}\left(\begin{array}{cc}  8 p^{(1)}_i+4  p^{(3)}_i &   8 p^{(2)}_i+4  p^{(3)}_i \\
6 p^{(1)}_i+ 6  p^{(3)}_i & 6 p^{(2)}_i+6  p^{(3)}_i \end{array}\right)$, and\\
 $\check N_i = \frac{1}{24}\left(\begin{array}{cc}  8 q^{(1)}_i+4  q^{(3)}_i &   8 q^{(2)}_i+4  q^{(3)}_i \\
6 q^{(1)}_i+ 6 q^{(3)}_i & 6 q^{(2)}_i+6 q^{(3)}_i \end{array}\right)$.
It follows that
$
\det\left(\check M_i\right) = \frac{p_i^{(3)}}{24}\left(p_i^{(1)}-p_i^{(2)}\right)$. Hence $\check M_i$ is invertible if and only if $p_i^{(1)}\neq p_i^{(2)}$ a.s.

\end{example}

\begin{remark} What happens if Hypothesis \ref{hyp:inv} is not met? In this case, one proceeds almost as before, reducing the dimension of $\tilde M_i$ instead. For example, taking $p_i^{(1)}=p_i^{(2)}$ in Example \ref{ex:weird}, i.e., Games 1 and 2 are the same, $
\check M_i = \left(\begin{array}{cc}  \alpha_i &   \alpha_i \\
\beta_i & \beta_i \end{array}\right)$,
 $\check N_i = \left(\begin{array}{cc}  \frac{1}{2}-\alpha_i &   \frac{1}{2}-\alpha_i \\
\frac{1}{2}-\beta_i & \frac{1}{2}-\beta_i \end{array}\right)$, where
$\alpha_i = \frac{1}{24}\left( 8 p^{(1)}_i+4  p^{(3)}_i\right)$, and $\beta_i  = \frac{1}{24}\left(
6 p^{(1)}_i+ 6  p^{(3)}_i \right)$. Hence, $p_i = \alpha_i+\beta_i  = \frac{1}{24}\left( 14 p^{(1)}_i+ 10  p^{(3)}_i\right) < 1$.
Thus Hypothesis \ref{hyp:inv} does not hold. However, setting $\check f_{i\ell} = \left(\begin{array}{c}  P_{1i}(\tilde \tau_\ell<\infty)  \\
P_{2i}(\tilde \tau_\ell<\infty)\end{array}\right)$, one has  $ \check f_{\ell\ell} = {\bf 1}_2$ and
\begin{equation}\label{eq:exweird}
\check f_{i\ell} = \check M_i \check f_{i+1,\ell}+\check N_i \check f_{i-1,\ell}, \qquad i\neq \ell.
\end{equation}
Also, $g_{i\ell} = P_{3i}(\tilde \tau_\ell<\infty) = \frac{1}{2}{\bf 1}_2 ^\top \check f_{i\ell}$.
It then follows from \eqref{eq:exweird} that
\begin{equation}\label{eq:exweird1}
g_{i\ell} = p_i g_{i+1,\ell}+(1-p_i) g_{i-1,\ell}, \qquad i\neq \ell,\quad g_{\ell}=1,
\end{equation}
and $\check f_{i\ell} = \left(\begin{array}{c}  2\alpha_i g_{i+1,\ell}+ (1-2\alpha_i)g_{i-1,\ell}\\
2\beta_i g_{i+1,\ell}+ (1-2\beta_i)g_{i-1,\ell} \end{array}\right)$.  Hence we are back to the first case considered, i.e., the rank $1$ case.
\end{remark}
\appendix

\section{Oseledec's multiplicative ergodic
theorem}\label{app:oseledec}

The following statement of the celebrated Oseledec's Theorem is taken from \citet{Walters:1982}.

\begin{theorem}\label{thm:oseledec}
Let $A_1,A_2,\ldots$ be a stationary ergodic sequence of $d\times d$
matrices such that $E \left\{\log^+ \left(\|A_1\|\right)\right\}
<\infty$. Then there exists constants $-\infty \le \lambda_d \le
\lambda_{d-1} \le \cdots \le \lambda_1 <\infty$ with the following
properties:
\begin{enumerate}
\item[(a)] With probability 1, the random sets
$$
{\cV}_q = {\cV}_q(e) = \left\{v\in \mathbb{R}^d:
\lim_{n\to\infty} n^{-1}\log\left(\|A_n(e) \cdots A_1(e)
v\| \le \lambda_q \right)\right\}
$$
are subspaces. The map $e \mapsto {\cV}_q(e)$ is
measurable and if $T$ is the measure preserving map for
which $A_i(Te) = A_{i+1}(e)$, then ${\cV}_q(Te)
=
A_1(e){\cV}_q(e)$.\\

\item[(b)] ${\rm Dim}({\cV}_q) = {\rm card}\{i: \lambda_i \le
\lambda_q\}$.\\

\item[(c)] Set $ {\cV}_{d+1} = \{0\}$ and let $i_1 =1 < i_2 \cdots < i_{p+1}=
d+1$ be the unique indices at which $\lambda_i$ jumps, i.e.,
$\lambda_1 = \lambda_2 = \cdots = \lambda_{i_2-1} > \lambda_{i_2}
\cdots$. Then for $v\in {\cV}_{i_{s-1}}\setminus {\cV}_{i_s}$, one
has
$$
\lim_{n\to\infty} n^{-1}\log\left(\|A_n \cdots A_1v\|\right) =
\lambda_{i_{s-1}}, \quad 2\le s\le p+1.
$$
\item[(d)] The sequence of matrices $\left(A_1^\top \cdots A_n^\top
A_n \cdots A_1\right)^{1/(2n)}$ converges almost surely to a limit
matrix $B$ with eigenvalues $\mu_1=e^{\lambda_1}, \ldots, \mu_d =
e^{\lambda_d}$. The orthogonal complement of ${\cV}_{i_s}$ in
${\cV}_{i_{s-1}}$ is the eigenspace of $B$ corresponding to
$\mu_{i_{s-1}}$.\\

\item [(e)] If $\limsup_{n\to\infty} n^{-1}E\left\{ \log\left(\|A_n \cdots
A_1\|\right)\right\} >0$ and $\det(A_1)=1$ with probability 1, then
$\lambda_d<0<\lambda_1$ and ${\cV}_d$, the subspace corresponding to
$\lambda_d$ is a proper  nonempty subspace of $\mathbb{R}^d$.
\end{enumerate}

\end{theorem}


\section{Proofs of the main results}

\subsection{Proof of Lemma \ref{lem:main}}\label{pf:mainlemma}

\begin{proof}\label{pf:mainlem}  Set ${\mathcal{A}}_\ell = \{X_n
\le \ell \; i.o\}$. First,
\begin{eqnarray*}
P^e_{\alpha i}({\mathcal{A}}_\ell) &=& \sum_{\beta=1}^m
P^e_{\alpha i}(\{\tau_\ell <\infty,G_{\tau_\ell}=\beta\}\cap
{\mathcal{A}}_\ell)+P^e_{\alpha i}(\{\tau_\ell
=\infty\}\cap {\mathcal{A}}_\ell)\\
&=& \sum_{\beta=1}^m  P^e_{\alpha i}(\tau_\ell <\infty, G_{\tau_\ell} =
\beta)P^e_{\beta \ell}({\mathcal{A}}_\ell)+P^e_{\alpha
i}(\{\tau_\ell =\infty\}\cap {\mathcal{A}}_\ell).
\end{eqnarray*}
As a result, for  $i<\ell$, one obtains
\begin{equation}\label{eq:i<l}
P^e_{\alpha i}({\mathcal{A}}_\ell) = \sum_{\beta=1}^m
P^e_{\alpha i}(\tau_\ell <\infty, G_{\tau_\ell} = \beta)P^e_{\beta
\ell}({\mathcal{A}}_\ell)+ P^e_{\alpha i}(\tau_\ell =\infty),
\end{equation}
while if $i>\ell$, then
\begin{equation}\label{eq:i>l}
P^e_{\alpha i}({\mathcal{A}}_\ell) = \sum_{\beta=1}^m
P^e_{\alpha i}(\tau_\ell <\infty, G_{\tau_\ell} = \beta)P^e_{\beta
\ell}({\mathcal{A}}_\ell).
\end{equation}

First, suppose that $P^e_{\alpha i}(\mathcal{A}_\ell)=0$ for
any $(\alpha,i)\in \{1,\ldots,m\}\times \dZ$. Then, according to
\eqref{eq:i<l},
 $P^e_{\alpha i}(\tau_\ell <\infty)=1$ for any $\alpha\in\{1,\ldots,m\}$ and any $i<\ell$.
Next, if $i>\ell$, then $P^e_{\alpha i}(\mathcal{A}_\ell)=0$
implies that $P^e_{\alpha i}(\tau_\ell<\infty)<1$. Hence
\eqref{eq:condlem} holds
true.\\

 Suppose now that \eqref{eq:condlem} holds
true. Then combining \eqref{eq:i<l} and \eqref{eq:i>l}, one obtains
\begin{equation}\label{eq:eq1}
P^e_{\alpha i}({\mathcal{A}}_\ell) =\sum_{\beta=1}^m
P^e_{\alpha i}(\tau_\ell <\infty, G_{\tau_\ell} = \beta)P^e_{\beta
\ell}({\mathcal{A}}_\ell), \quad \alpha\in \{1,\ldots,m\}, i\neq
\ell.
\end{equation}
Therefore, to complete the proof, it suffices to show that
$P^e_{\alpha \ell}({\mathcal{A}}_\ell)=0$ for any $\alpha\in
\{1,\ldots,m\}$. To this end, note that using \eqref{eq:eq1}, one
gets
\begin{eqnarray} P^e_{\alpha \ell}({\mathcal{A}}_\ell) &=& \sum_{\beta=1}^m
P^e_{\alpha \ell}(\{G_1=\beta\} \cap {\mathcal{A}}_\ell) \nonumber\\
&=&\sum_{\beta=1}^m Q_{\alpha \beta}\left[ p_\ell^{(\beta)}
P^e_{\beta,\ell+1}({\mathcal{A}}_\ell)+ q_\ell^{(\beta)}
P^e_{\beta,\ell-1}({\mathcal{A}}_\ell)\right] \nonumber\\
&=&\sum_{\beta=1}^m \sum_{\gamma=1}^m Q_{\alpha \beta} p_\ell^{(\beta)} P^e_{\beta,\ell+1}(\tau_\ell <\infty, G_{\tau_\ell}=\gamma)
P^e_{\gamma\ell}({\mathcal{A}}_\ell) \nonumber\\
&& \qquad + \sum_{\beta=1}^m \sum_{\gamma=1}^m Q_{\alpha \beta}
q_\ell^{(\beta)} P^e_{\beta,\ell-1}(\tau_\ell <\infty,
G_{\tau_\ell}=\gamma) P^e_{\gamma\ell}({\mathcal{A}}_\ell) \nonumber\\
&=& \sum_{\gamma=1}^m P^e_{\alpha \ell}(\tau_\ell <\infty,
G_{\tau_\ell}=\gamma)
P^e_{\gamma\ell}({\mathcal{A}}_\ell),\label{eq:main11}
\end{eqnarray}
since
\begin{eqnarray}\label{eq:fll}
\qquad  P^e_{\alpha \ell}(\tau_\ell <\infty, G_{\tau_\ell} = \gamma) &=&
\sum_{\beta=1}^m Q_{\alpha\beta}  p_\ell^{(\beta)}
P^e_{\beta,\ell+1}(\tau_\ell <\infty,
G_{\tau_\ell} =\gamma) \\
&& +\sum_{\beta=1}^m Q_{\alpha\beta} q_\ell^{(\beta)}
P^e_{\beta,\ell-1}(\tau_\ell <\infty, G_{\tau_\ell}=\gamma).\nonumber
\end{eqnarray}
By hypothesis, $P^e_{\alpha, \ell+1}(\tau_\ell <\infty)<1 $ for
any $\alpha\in \{1,\ldots,m\}$, so by \eqref{eq:fll} it follows that
$ f^{e,*}_{ll} = \max_{1\le \alpha\le m}P^e_{\alpha
\ell}(\tau_\ell <\infty) < 1$. Further set $\mathcal{P}_\ell =
\max_{1\le \alpha\le m}P^e_{\alpha \ell}({\mathcal{A}}_\ell)$.

Then, from \eqref{eq:main11}, $
\mathcal{P}_\ell = \max_{1\le \alpha \le m} P^e_{\alpha
\ell}({\mathcal{A}}_\ell) \le f^{e,*}_{ll} \mathcal{P}_\ell$,
so $\mathcal{P}_\ell\left(1-f^{e,*}_{ll}\right)\le 0$. Since
$f^{e,*}_{ll}<1$, it follows that $\mathcal{P}_\ell=0$. Hence
the result.
\end{proof}

\subsection{Proof of Proposition \ref{prop:exit}}\label{app:pf_propexit}
\begin{proof}
Let $e$ be given.
First, \eqref{eq:sum} is obviously true for $\ell=1$. If
\eqref{eq:sum} is not true for some $\ell>1$, then for some $\alpha
\in \{1,\ldots,m\}$, and some $|i|<\ell$,
$$
P^e_{\alpha i}(|X_n|< \ell \mbox{ for all }n\ge 1 )>0.
$$
Next, for  $|i|<\ell$ and $\alpha \in \{1,\ldots,m\}$, set
$g_{\alpha i}(e) = P^e_{\alpha i}(|X_n|< \ell \mbox{ for all }n\ge
1 )$. It then follows that for any $\alpha \in \{1,\ldots,m\}$, and
$|i|<\ell-1$,
$$
g_{\alpha i} = \sum_{\beta=1}^m Q_{\alpha \beta} \left(p^{(\beta
)}_{i}g_{\beta, i+1} +q^{(\beta )}_{i}g_{\beta,
i-1}\right),
$$
while
$
g_{\alpha, \ell-1} = \sum_{\beta=1}^m Q_{\alpha \beta} q_{\ell-1}^{(\beta)}g_{\beta \ell-2}$,
and
$
g_{\alpha, -\ell+1} = \sum_{\beta=1}^m Q_{\alpha \beta} p_{-\ell+1}^{(\beta)} g_{\beta ,-\ell+2}$.
It follows that for every $e$, there is a sub-stochastic matrix $\tilde
P_{\alpha i,\beta j}(e)$ on $S=\{1,\ldots,m\}\times
\{-\ell+1,\ldots, \ell-1\}$, so that $g_{\alpha i} = \sum
_{(\beta,j)\in S} \tilde P_{\alpha i,\beta j} g_{\beta j}$,
for any $(\alpha,i)\in S$. Note that $\left(\tilde P
\mathbf{1}\right)_{\alpha i} = 1$ if $|i|<\ell-1$, while
$\left(\tilde P \mathbf{1}\right)_{\alpha,\ell-1} = \sum_{\beta=1}^m
Q_{\alpha \beta} q_{\ell-1}^{(\beta)} <1$, and $\left(\tilde P
\mathbf{1}\right)_{\alpha,-\ell+1} =\sum_{\beta=1}^m Q_{\alpha
\beta} p_{\ell+1}^{(\beta)}<1$. Next, $\left(\tilde P^2
\mathbf{1}\right)_{\alpha i}<1$ if $i \in
\{-\ell+1,-\ell+2,\ell-2,\ell-1\}$, and by induction, $\left(\tilde
P^k \mathbf{1}\right)_{\alpha i}<1$ if $i \in
\{-\ell+1,\ldots,\ell+k,\ell-k, \ldots,\ell-1\}$. Therefore, for any $e$, there
exists $c = c(e)\in (0,1)$ so that $\tilde P^{\ell} \mathbf{1} \le c
\mathbf{1}$. As a result, using \citet[Theorem
8.4]{Billingsley:1995}, one obtains that $g = \lim_{n\to\infty}
\tilde P^n \mathbf{1} = 0$,  contradicting the hypothesis that
$g_{\alpha i}>0$ for some $(\alpha,i)\in S$. Hence the result.
\end{proof}

%

\subsection{Proof of Proposition
\ref{prop:time3}}\label{pf:proptime3}
One only proves the proposition for $i<\ell$, the
proof of the case $i>\ell$ being similar.
If $\|f_{i\ell}(e)\|=1$ for $i=\ell-1$, then for some $\alpha\in \{1,\ldots,m\}$,
$P^e_{\alpha, \ell-1}(\tau_\ell <\infty)=1$. Hence,  for any
$\beta$ so that $Q_{\alpha\beta}>0$, one has
$P^e_{\beta,\ell-2}(\tau_\ell <\infty)=1$, according to Proposition \ref{prop:time}. Therefore $\|f_{\ell-2,\ell}(e)\|=1$. Next, if $\|f_{i\ell
}(e)\|=1$ for some $i<\ell-1$, then Proposition \ref{prop:time} implies that
$\|f_{i\pm 1,\ell}(e)\|=1$. This proves the first part of the
proposition. Suppose now that $P^e_{\alpha i}(\tau_\ell<\infty)=1$, $e$
a.s. Since $G_n$ is irreducible,  for a given
$\beta$, one can find $n\ge 1$ so that $Q^n_{\alpha\beta}>0$. As a
result, it follows from Proposition \ref{prop:time}  that
$P^e_{\beta, i-n} (\tau_\ell <\infty)=1$ a.s. Hence, using
stationarity, $P^e_{\beta, i} (\tau_{\ell+n} <\infty)=1$, $e$ a.s.,
entailing that $P^e_{\beta, i} (\tau_{\ell} <\infty)=1$, $e$ a.s.
 \qed

\subsection{Proof of Proposition
\ref{prop:time4}}\label{pf:proptime4}
For any $i<j <\ell$,
$$
\left(f_{i\ell}(e)\right)_\alpha = P^e_{\alpha i}(\tau_\ell
<\infty) = \sum_{\beta=1}^m P^e_{\alpha i}(\tau_j <\infty,
G_{\tau_j}=\beta)P^e_{\beta j}(\tau_\ell < \infty).
$$
As a result, for any $e$, $\|f_{i\ell}(e)\| \le \|f_{ij}(e)\|
\|f_{j\ell}(e)\|$, showing that the logarithm of the sequence is
subadditive. Then \eqref{eq:prodergo} follows easily. Next,
$\|f_{ij}(e)\|= \|f_{0,j-i}(T^i e)\|$.
This proves that $\|f_{\ell-k,\ell-k+1}\|$, $k\ge 1$, is a
stationary ergodic sequence, so by using the Subadditive Ergodic
Theorem \citep{Durrett:2010}, there exists a constant $ \gamma_-$ so
that for almost every $e \in E$,
$$
\lim_{n\to\infty}
\frac{1}{n}\log{\|f_{\ell-n,\ell}(e)\|}=\gamma_- \le
E\left\{\log{\|f_{0,1}\|}\right\},
$$
proving the first part.
Next, if $i>j>\ell$, then
$$
\left(f_{i\ell}(e)\right)_\alpha = P^e_{\alpha i}(\tau_\ell
<\infty) = \sum_{\beta=1}^m P^e_{\alpha i}(\tau_j <\infty,
G_{\tau_j}=\beta)P^e_{\beta j}(\tau_\ell < \infty),
$$
so again $\|f_{i\ell}\| \le \|f_{ij}\|
\|f_{j\ell}\|$. The rest of the proof follows along the same lines as the
previous case $i<j <\ell$.
 \qed

\subsection{Auxiliary results}\label{aux}
We next give the formula for the $k$-th visiting time of $X_n$ to a
state $\ell\in\dZ$. To this end, let $\tau_\ell^{(k)}$ denotes the time
of the k-th visit to $\ell$ and define $\cN_\ell$ as the number of visits to site $\ell$.

\begin{proposition}\label{prop:kthvisit}
For $\ell$ given, set $(\cU_\ell)_{\alpha\beta} =
\left( f^{(\beta)}_{\ell\ell}\right)_\alpha$ while recalling that,
for every $\alpha,\beta\in\{1,\ldots,m\}$, $i\in \dZ$,
$\left( f^{(\beta)}_{i\ell}\right)_\alpha(e)
=P^e_{\alpha i}(\tau_\ell<\infty,G_{\tau_\ell}=\beta)$. Then for any $k\ge 1$,
\begin{equation}\label{eq:kthvisit1}
P^e_{\alpha i}\left(\tau_\ell^{(k)} < \infty,
G_{\tau_\ell^{(k)}}=\beta\right) = \sum_{\gamma=1}^m
\left( f^{(\gamma)}_{i\ell}\right)_\alpha (e)( \cU_\ell^{k-1}(e))_{\gamma\beta}
\end{equation}
and
\begin{equation}\label{eq:kthvisit2}
P^e_{\alpha i}\left(\cN_\ell \ge k \right)=
P^e_{\alpha i}\left(\tau_\ell^{(k)} < \infty\right) =
\sum_{\gamma=1}^m \sum_{\beta=1}^m \left( f^{(\gamma)}_{i\ell}\right)_\alpha (e)
( \cU_\ell^{k-1}(e))_{\gamma\beta},
\end{equation}
with $ \cU^0 = I$.
\end{proposition}

\begin{proof}\label{pf:kthvisit}
Set $\tau_\ell^{(0)} = 0$. Then for any $\alpha,\beta\in\{1,\ldots,m\}$
and any $i\in \dZ$, one has
\begin{multline*}
P^e_{\alpha i}\left(\tau_\ell^{(k)} < \infty,
G_{\tau_\ell^{(k)}}=\beta\right) \\ = \sum_{\gamma=1}^m P^e_{\alpha i}\left(\tau_\ell^{(k)} < \infty, G_{\tau_\ell^{(k)}}=\beta, \tau_\ell^{(k-1)} < \infty,G_{\tau_\ell^{(k-1)}}=\gamma\right)\\
=\sum_{\gamma=1}^m P^e_{\alpha i}\left(\tau_\ell^{(k-1)} <
\infty,G_{\tau_\ell^{(k-1)}}=\gamma\right) P^e_{\gamma\ell}
\left(\tau_\ell < \infty, G_{\tau_\ell}=\beta\right).
\end{multline*}
The result then follows by induction on $k$.
\end{proof}

\subsection{Proof of Lemma \ref{lem:main0}}\label{app:pf-lemma2}
\begin{proof}
Set ${\mathcal{A}}_\ell = \{X_n
\le \ell \; i.o\}$.
First, $P^e_{\alpha i}(\lim_{n\to\infty} X_n = +\infty) =1$ for
every $(\alpha,i)\in \{1,\ldots,m\}\times \dZ$ if and only if for
every $(\alpha,i,\ell)\in \{1,\ldots,m\}\times \dZ^2$,
$P^e_{\alpha i}(\cA_\ell)=0$. By Lemma \ref{lem:main}, this is
equivalent to \eqref{eq:condthm+1} and \eqref{eq:condthm+2}. This
proves (i). Next, (ii) follows from (i) applied to $-X_n$. To prove
(iii), note that $P^e_{\alpha i}(\liminf_{n\to\infty} X_n =
-\infty) =1$ for every $(\alpha,i)\in \{1,\ldots,m\}\times \dZ$ if
and only if for every $(\alpha,i,\ell)\in \{1,\ldots,m\}\times
\dZ^2$, $P^e_{\alpha i}(\cA_\ell)=1$. The latter implies that
$P^e_{\alpha i}(\tau_\ell <\infty) = 1$ whenever $i > \ell$. Using this result with $-X_n$, one obtains that
$P^e_{\alpha i}(\limsup_{n\to\infty} X_n = +\infty) =1$ for
every $(\alpha,i)\in \{1,\ldots,m\}\times \dZ$ implies that for every
$(\alpha,i,\ell)\in \{1,\ldots,m\}\times \dZ^2$, $P^e_{\alpha
i}(\tau_\ell <\infty)=1$ whenever $i<\ell$.  Then \eqref{eq:fll} yields
that $P^e_{\alpha \ell}(\tau_\ell<\infty)=1$. (Irreducibility has not been used yet.)

To complete the proof, suppose now that \eqref{eq:condthm+-} holds
true.
Proposition \ref{prop:time} with $i=\ell$ combined with Proposition \ref{prop:time3}
first implies that $\|f_{i\ell}(e)\|=1$ for any $i$ and $\ell$. With the additional condition that
the Markov chain $G_n$ is irreducible, Proposition \ref{prop:time3} further implies that
$P^e_{\alpha i}(\tau_\ell<\infty)=1$ for any $(\alpha,i,\ell)\in
\{1,\ldots,m\}\times \dZ^2$. By using \eqref{eq:eq1}, it suffices to
show that $P^e_{\alpha \ell}(\cA_\ell) = 1$ for every
$\alpha\in\{1,\ldots,m\}$. Now, by hypothesis, $ \cU_\ell$,  defined
by $( \cU_\ell)_{\alpha\beta} = \left( f^{(\beta)}_{\ell\ell}\right)_\alpha$,
$\alpha,\beta\in\{1,\ldots,m\}$, is a stochastic matrix, so $
P^e_{\alpha \ell}\left(\tau_\ell^{(k)} < \infty\right)=1$ for every
$k\ge 1$, by using \eqref{eq:kthvisit2}. As a result,
$P^e_{\alpha \ell}(\cA_\ell) \ge P^e_{\alpha \ell}(\cN_\ell
= \infty) =1$.
\end{proof}

\subsection{Proof of Theorem \ref{thm:main}}\label{pf:mainthm}

\begin{proof} One first proves that  $\max(\gamma_+,\gamma_-)=0$. In
fact, it follows from Proposition \ref{prop:time4} that
$\max(\gamma_+,\gamma_-)\le 0$. One now shows by contradiction that
it is impossible to have $\gamma_+<0$ and $\gamma_-<0$. So suppose
that $\max(\gamma_+,\gamma_-)<0$. Since \eqref{eq:sum} holds by
Proposition \ref{prop:exit}, it follows that
$$
\|f_{0,\ell}(e)\|+\|f_{0,-\ell}(e)\| \ge  P^e_{\alpha
0}(\tau_\ell < \infty ) + P^e_{\alpha 0}(\tau_{-\ell} < \infty) \ge
1,
$$
which is impossible since $\gamma_+<0$ and $\gamma_-<0$. Hence,
$\max(\gamma_+,\gamma_-)=0$.

It follows from Propositions \ref{prop:time3}-\ref{prop:time4}
that if  $\gamma_+ <0$, then  for any $\alpha\in \{1,\ldots,m\}$ and
any $i> \ell$, $P^e_{\alpha i}(\tau_\ell <\infty)<1$, $e$ a.s.
Similarly, if $\gamma_-<0$, then for any $\alpha\in \{1,\ldots,m\}$
and any $i< \ell$, $P^e_{\alpha i}(\tau_\ell <\infty)<1$, $e$ a.s.

Next, if $\gamma_-=0$, then $\|f_{01}\|=1$ a.s. by
Proposition \ref{prop:time4}. It then follows from the irreducibility of $(G_n,X_n)$
 that $P^e_{\alpha i}(\tau_\ell < \infty)=1$, $e$ a.s.,
for any $\alpha\in \{1,\ldots,m\}$ and any $i<\ell$. In fact, if for a given $e$, $P^e_{\alpha i}(\tau_\ell < \infty)=1$, then it follows that $P^e_{\beta i}(\tau_\ell < \infty)=1$, for all $\beta\in \{1,\ldots,m\}$.

Similarly, if
$\gamma_+=0$, then $\|f_{10}\|=1$ a.s. by Proposition
\ref{prop:time4}, so the irreducibility of $(G_n,X_n)$ entails
that $P^e_{\alpha i}(\tau_\ell <\infty)=1$, $e$ a.s.,  for any
$\alpha\in \{1,\ldots,m\}$ and any $i>\ell$. As a result, using
Lemma \ref{lem:main0}, one gets (1), (2), and (3).
\end{proof}

\subsection{Proof of Theorem \ref{thm:oseledecfull}}\label{pf:thmoseledec}

\begin{proof}
 By Oseledec's Theorem (Theorem \ref{thm:oseledec}), there exists
 constants  $-\infty \le \bar\lambda_{2m} \le
\bar\lambda_{2m-1} \le \cdots \le \bar\lambda_1 <\infty$ with the
following properties:
\begin{enumerate}
\item[(a)] With probability 1, the random sets
$$
\bar V_q = \bar V_q(e) = \left\{v\in \mathbb{R}^{2m}:
\lim_{n\to\infty} n^{-1}\log\left(\|A_{n-1}(e) \cdots
A_0(e) v\| \le \bar\lambda_q \right)\right\}
$$
are linear subspaces. The map $e \mapsto {\bar V}_q(e)$ is
measurable and  ${\bar
V}_q(Te) =
A_0(e){\bar V}_q(e)$.\\

\item[(b)] ${\rm Dim}({\bar V}_q) = {\rm card}\{i: \bar\lambda_i \le
\bar\lambda_q\}$.\\

\item[(c)] Set $ {\bar V}_{2m+1} = \{0\}$ and let $i_1 =1 < i_2 \cdots < i_{p+1}=
2m+1$ be the unique indices at which $\bar\lambda_i$ jumps, i.e.,
$\bar\lambda_1 = \bar\lambda_2 = \cdots = \bar\lambda_{i_2-1} > \bar\lambda_{i_2}
\cdots$. Then for $v\in {\bar V}_{i_{s-1}}\setminus {\bar V}_{i_s}$,
one has
$$
\lim_{n\to\infty} n^{-1}\log\left(\|A_{n-1} \cdots A_0 v\|\right) =
\bar\lambda_{i_{s-1}}, \quad 2\le s\le p+1.
$$
\end{enumerate}

Since each $A_i$ is invertible, it then follows from Oseledec's
Theorem that for any $\ell\in \dZ$, the product $A_{n-1} \cdots A_0$
can be replaced with the product $A_{\ell+n} \cdots A_{\ell+1}$ and
one still gets the same constants $\bar \lambda_j$ and the
dimensions of the associated subspaces are exactly the same.

Let  $\ell$ be given. For $i>\ell$, set ${\bar v}_{i}^{(\beta)} =
\left[\begin{array}{c}
 \tilde f^{(\beta)}_{i\ell}\\ \tilde f^{(\beta)}_{i-1,\ell}\end{array}\right]$, $i>\ell$, $\beta\in\{1,\ldots,m\}$. Using \eqref{eq:main1},
one can write
\begin{equation}\label{eq:vbarbeta} {\bar
v}_{i+1}^{(\beta)} = A_i {\bar v}_i^{(\beta)} = A_i \cdots
A_{\ell+1} {\bar v}_{\ell+1}^{(\beta)},\quad i
>\ell.
\end{equation}

Since $\tilde f^{(\beta)}_{\ell\ell} = e^{(\beta)}$ for every $\beta\in
\{1,\ldots,m\}$, it follows that the vectors ${\bar
v}_{\ell+1}^{(\beta)}$ are linearly independent, and they belong to
$\bar V_0$. As a result, $\bar d_0 = \dim{\bar{V}_{0}} \ge m$. Also,
$A_i \un =\un$, so
$\left[\begin{array}{c} \mathbf{1} \\
\mathbf{1}\end{array}\right]\in \bar{V}_{0}$. Therefore $\bar d_0 > \bar d_{0-}$.

\subsubsection{Proof of (1)}\label{sssec:Qinv+}

Suppose $\bar d_0=m$. Then for almost every environment, there
is a  random vector $\theta\in \dR^m$ such that $\sum_{\beta=1}^m
\theta_\beta v_{\ell+1}^{(\beta)} =\un$. In particular,
$\sum_{\beta=1}^m \theta_\beta \tilde f^{(\beta)}_{\ell+1,\ell}
=
\mathbf{1}$ a.s.  and $ \sum_{\beta=1}^m \theta_\beta
\tilde f^{(\beta)}_{\ell,\ell} = \mathbf{1}$ a.s., which implies $ \sum_{\beta=1}^m \theta_\beta
\tilde f^{(\beta)}_{\ell,\ell} =  \sum_{\beta=1}^m \theta_\beta
e^{(\beta)} =\theta=\mathbf{1}$.
It then follows that $\sum_{\beta=1}^m \bar v_{\ell+1}^{(\beta)} = \un$, and
one can deduce from \eqref{eq:vbarbeta} that for any $i>\ell$,
$\sum_{\beta=1}^m \bar v_{i}^{(\beta)} = \un$. Hence, for
any $i>\ell$, one has $\tilde f_{il}(e) = f_{i\ell} (e) = \bf{1}$,
and consequently, for any
$\alpha\in \{1,\ldots,m\}$,
$$
\left(\tilde f_{il}(e)\right)_\alpha =  P^e_{\alpha i}(\tau_\ell <\infty ) = 1 \qquad e \mbox{ a.s.}.
$$
As a result, $\gamma_+=0$. This completes the proof of (1).

\subsubsection{Proof of (2)}  The proof proceeds somewhat along the lines of \citet{Key:1984}.
First note that for any (possibly random) element $h_\ell \in \bar V_{0-}$ with $V_{0-}$ as in Section \ref{sec:hitting},
we can define a sequence
$h_i = A_i h_{i-1}$ for $i>\ell$ with the property that  $h_i $ has the form $h_i =  \left[\begin{array}{c} z_{i+1} \\
z_i\end{array}\right]$ for some (unique, possibly random) sequence $z_\ell, z_{\ell+1}, \ldots\in \mathbb{R}^m$,   because of the structure of $A_i$ defined in Section \ref{sec:hitting}.
Set $g(\alpha,i) =(z_i)_\alpha$, $\alpha\in \{1,\ldots,m\}$, $i\ge \ell$
and notice that, because of $h_\ell \in \bar V_{0-}$,
$\lim_{n\to\infty} g(\alpha,n)=0$ ensues.

Let $(G_n,Y_n)$ be the Markov chain starting at $(\alpha,i)$ associated
with $(G_n,X_n)$ but absorbed on $\{1,\ldots,m\}\times \{\ell\}$,
and set $\cM_n = g(G_n,Y_n)$.
Since $\lim_{n\to\infty} g(\alpha,n)=0$, $\cM$ is a bounded
martingale with respect to its natural filtration
(provided the initial sigma field is enlarged for $z_i$ to be measurable with respect to it)
started at $ \cM_0 = g(\alpha,i)$.  Because of its random walk
nature, either $(G_n,Y_n)$  is absorbed on the boundary or $\limsup_{n\to\infty}Y_n
= +\infty$. As a result, by the martingale convergence theorem, it
follows that $\cM_n \to \cM_\infty$, so
$g(\alpha,i) = E(\cM_0) = E(\cM_\infty)$ for any $i>\ell$ and any $\alpha\in\{1,\ldots,m\}$.\\

Suppose first that $\bar d_{0-}\ge  m$. Let ${\bar v}_{-1}, \ldots, \bar v_{-m}$ be a set of (possibly random)
linearly independent vectors in $\bar V_{0-}$. The last $m$ components of
${\bar v}_{-1}, \ldots, \bar v_{-m}$ are linearly independent. If
not, there exists a member $h_\ell \neq 0$ of $\bar
V_{0-}$ so that its last $m$ components are zero.
In this case, it follows that $\cM_\infty =0$, so $g(\alpha,i) = E(\cM_0) = E(\cM_\infty) = 0$.
Since the latter is true for any $i>\ell$ and any $\alpha\in\{1,\ldots,m\}$, one may conclude that $h_\ell\equiv 0$,
contradicting the assumption. Thus, the last $m$
components of ${\bar v}_{-1}, \ldots, \bar v_{-m}$ are linearly
independent. Because of this, there exists $h_\ell \in \bar V_{0-}$ such that
its last $m$ components are $1$. We set $z_\ell={\bf 1}$ henceforth.

Recall that
for any $\alpha\in \{1,\ldots,m\}$,
$\left(\tilde f_{i\ell}(e)\right)_\alpha = P^e_{\alpha
i}\left(\tilde \tau_\ell<\infty\right)$ and define  ${w}_{i} = \left[\begin{array}{c}
 \tilde f_{i+1,\ell} \\ \tilde f_{i\ell}\end{array}\right]$, $i\ge \ell$. By \eqref{eq:main1},
we get $ w_i = A_i w_{i-1}$, $i >\ell$. Note also that
$\tilde f_{\ell\ell}={\bf 1}$ by definition.
Now set $\tilde \cM_n = \left(\tilde f_{Y_n,\ell}\right)_{G_n}$; $\tilde \cM$ also forms a bounded
martingale with $\tilde \cM_0 = \left(\tilde f_{i\ell}(e)\right)_\alpha$,
and
\begin{eqnarray*}
\left(\tilde f_{i\ell}\right)_\alpha &= & E(\tilde \cM_0) = E(\tilde \cM_\infty) = E\{\tilde \cM_\infty
\I(\tilde \tau_\ell<\infty)\}+E\{\tilde \cM_\infty \I(\tilde \tau_\ell=\infty)\}\nonumber\\
&=& \left(\tilde f_{i\ell}\right)_\alpha+ E\{\tilde \cM_\infty \I(\tilde \tau_\ell
=\infty)\},
\end{eqnarray*}
since $\tilde \cM_\infty = 1$ on $\{\tilde \tau_\ell <\infty\}$. As  a result, $
E\{\tilde \cM_\infty \I(\tilde \tau_\ell =\infty)\}=0$.
The bounded martingale $\tilde \cD_n = \tilde \cM_n - \cM_n $ satisfies
\begin{eqnarray*}
\left(\tilde f_{i\ell} - z_i\right)_\alpha &= & E(\tilde \cD_0) = E(\tilde \cD_\infty) =
E\{\tilde \cD_\infty \I(\tilde \tau_\ell=\infty)\} \\
 &=&
E\{\tilde \cM_\infty \I(\tilde \tau_\ell =\infty)\} - 0 = 0,
\end{eqnarray*}
since $\tilde f_{\ell\ell}-z_\ell= 0$ and
$\limsup_{n\to\infty}Y_n = +\infty$ on $\{\tilde \tau_\ell = \infty\}$
implies $\lim_{n\to\infty}g(G_n,Y_n)\I(\tilde \tau_\ell=\infty) = 0$.

Hence $\tilde f_{i\ell} = z_i$ for any $i\ge \ell$. It then follows that $\gamma_+ <0$, and
so $\lim_{i\to \infty} P^e_{\alpha i}(\tau_\ell <\infty )=0$, $e$ a.s.,  for any
$\alpha\in\{1,\ldots,m\}$.

To complete the proof of (2), suppose now that $\gamma_+<0$. Combined with
equation \eqref{eq:vbarbeta} this asumption implies
${\bar v}_{i}^{(\beta)} =
\left[\begin{array}{c}
 \tilde f^{(\beta)}_{i\ell}\\ \tilde f^{(\beta)}_{i-1,\ell}\end{array}\right]\in \bar V_{0-}$
 for every $i>\ell$ and $\beta\in\{1,\ldots,m\}$. Since the $m$ vectors
 ${\bar v}_{\ell+1}^{(\beta)}$ are linearly independent, we conclude $\bar d_{0-} \ge m$.

\subsubsection{Proofs of (3) and (4)}
They are similar to those of (1) and (2).
In fact, setting ${\tilde v}_{i}^{(\beta)} =
\left[\begin{array}{c}
 \tilde f^{(\beta)}_{i-1} \\ \tilde f^{(\beta)}_{i}\end{array}\right]$, $i<\ell$, $\beta\in\{1,\ldots,m\}$, and
 using \eqref{eq:main1},
one can write \begin{equation}\label{eq:vtildebeta} {\tilde
v}_{i-1}^{(\beta)} = B_i {\tilde v}_i^{(\beta)},\quad i <\ell,
\end{equation}
with $B_i =  \left(\begin{array}{cc} N_i^{-1} & -\sigma_i^{-1}\\
I & 0 \end{array}\right) $, $M_i$, $N_i$ as in \eqref{eq:MN} and
 $\sigma_i = M_i^{-1} N_i$. We have
$B_i = G A_i^{-1} G$, where $G = \left(\begin{array}{cc} 0 & I\\
 I & 0\end{array} \right)$ and  $A_i^{-1} =\left(\begin{array}{cc} 0 & I\\
 -\sigma_i^{-1} & N_i^{-1}\end{array} \right)$.
As a result,
 $B_{\ell-n} \cdots B_{\ell-1} =
 G (A_{\ell-1} \cdots A_{\ell-n})^{-1} G$.
\end{proof}

\subsection{Proof of Theorem \ref{thm:oseledecgen}}\label{pf:oseledecgen}

Recall that from \eqref{eq:main1},
$
\tilde f^{(\beta)}_{i\ell} = Q\Delta_i \tilde f^{(\beta)}_{i+1,\ell}+
Q(I-\Delta_i) \tilde f^{(\beta)}_{i-1,\ell}$, $i\neq \ell, \beta\in
\{1,\ldots,m\}$,
with  $\tilde f^{(\beta)}_{\ell\ell} = e^{(\beta)}$. Also, since $\tilde f_{i\ell} = \sum_{\beta=1}^m \tilde f^{(\beta)}_{i\ell}$, it follows that
$
\tilde f_{i\ell} = Q\Delta_i \tilde f_{i+1,\ell}+
Q(I-\Delta_i) \tilde f_{i-1,\ell}$, $i\neq \ell$,
with $\tilde f_{\ell\ell} = \mathbf{1}$.

Let $\tilde f^{(1,\beta)}_{i\ell}$ and $\tilde f^{(2,\beta)}_{i\ell}$ represent respectively the first $r$ components and the last $m-r$ components of $\tilde f^{(\beta)}_{i\ell}$. For simplicity, set $\check f^{(\beta)}_{i\ell} = \tilde f^{(1,\beta)}_{i\ell}$.

It then follows from the representation of $Q$ and \eqref{eq:main1} that for any $i\neq \ell$ and any $\beta\in \{1,\ldots,m\}$,
\begin{eqnarray}
\label{eq:prem1}
\check f^{(\beta)}_{i\ell} &=& \pi^{(1)}\Delta_i^{(1)}\check f^{(\beta)}_{i+1,\ell}+ \pi^{(2)}\Delta_i^{(2)}\tilde f^{(2,\beta)}_{i+1,\ell}\\
&&  +\pi^{(1)}\left(I-\Delta_i^{(1)}\right)\check f^{(\beta)}_{i-1,\ell}+ \pi^{(2)}\left(I-\Delta_i^{(2)}\right)\tilde f^{(2,\beta)}_{i-1,\ell},\nonumber\\
\qquad \tilde f^{(2,\beta)}_{i\ell} &=& \Theta\check f^{(\beta)}_{i\ell}. \label{eq:prem2}
\end{eqnarray}

 As a result, one can write, for any $i\neq \ell$ and any $\beta\in \{1,\ldots,m\}$,
\begin{equation}\label{eq:check_f_beta}
\tilde f^{(\beta)}_{i\ell} =  \left[\begin{array}{c} \check f^{(\beta)}_{i\ell} \\
\Theta  \check f^{(\beta)}_{i\ell}\end{array}\right].
\end{equation}
Furthermore, since $\tilde f_{\ell\ell} = \mathbf{1}$ and $\Theta \mathbf{1}_r = \mathbf{1}_{m-r}$, it follows from \eqref{eq:prem1}--\eqref{eq:check_f_beta} that  for any  $i\in \dZ$,
\begin{equation}\label{eq:check_f}
\tilde f_{i\ell} =  \left[\begin{array}{c} \check f_{i\ell} \\
\Theta  \check f_{i\ell}\end{array}\right].
\end{equation}

Next, using \eqref{eq:prem1}--\eqref{eq:check_f}, one can write, for any $|i-\ell|>1$ and any $\beta\in \{1,\ldots,m\}$,
\begin{equation}\label{eq:check_f_beta_2}
\check f^{(\beta)}_{i\ell} = \check M_i \check f^{(\beta)}_{i+1,\ell}+
\check N_i \check f^{(\beta)}_{i-1,\ell},
\end{equation}
while for any $i\neq \ell$,
\begin{equation}\label{eq:check_f2}
\check f_{i\ell} = \check M_i \check f_{i+1,\ell}+
\check N_i \check f_{i-1,\ell}.
\end{equation}

Also, for any $\beta\in \{1,\ldots,m\}$,
\begin{eqnarray}\label{eq:check_f3}
\check f^{(\beta)}_{\ell+1,\ell} &=& \check M_{\ell+1} \check f^{(\beta)}_{\ell+2,\ell}+ \pi (I-\Delta_{\ell+1})  e^{(\beta)}\\
\label{eq:check_f4}
 \check f^{(\beta)}_{\ell-1,\ell} &=& \pi \Delta_{\ell-1} e^{(\beta)} + \check N_{\ell-1} \check f^{(\beta)}_{\ell-2,\ell}
\end{eqnarray}

Now, set $h^{(\beta)}_{\ell\ell} =  \left[\begin{array}{c} e^{(\beta)}_r \\
\Theta  e^{(\beta)}_{r}\end{array}\right]$, where $e^{(\beta)}_{r}$ is the vector composed of the first $r$ components of $e^{(\beta)}$.
Next, taking linear combinations of the functions $f^{(\beta)}_{i\ell}$, $\beta\in \{1,\ldots,m\}$, one can define, for any $i\neq \ell$ and any  $\beta\in \{1,\ldots,r\}$,
\begin{equation}\label{eq:check_h1}
\check h^{(\beta)}_{i,\ell} =  \check M_{i}\check  h^{(\beta)}_{i+1,\ell}+ \check N_i \check  h^{(\beta)}_{i-1,\ell},
\end{equation}
where $\check h^{(\beta)}_{i\ell}$ is the first $r$ components of $h^{(\beta)}_{i\ell} = \left[\begin{array}{c} \check h^{(\beta)}_{i\ell} \\
\Theta  \check h^{(\beta)}_{i\ell} \end{array}\right]$. Also $\sum_{\beta=1}^r \check h^{(\beta)}_{\ell\ell} = \mathbf{1}_r$, and
$\sum_{\beta=1}^r \check h^{(\beta)}_{i\ell} = \check f_{i\ell} $,  for any  $i\in \dZ$.
Under Hypothesis \ref{hyp:inv},
one can then apply Oseledec's theorem. Define $\check d_{0-}$ as the dimension of the set $\check V_{0-}$ of vectors $v\in \dR^{2r}$ so that $\lim_{n\to\infty}\frac{1}{n}\log\left\|\check A_n \cdots \check A_\ell\right\| < 0$. Next, let  $\ell$ be given. For $i>\ell$, set ${\check v}_{i}^{(\beta)} =
\left[\begin{array}{c}
 \check h^{(\beta)}_{i\ell}\\ \check h^{(\beta)}_{i-1,\ell}\end{array}\right]$, $i>\ell$, $\beta\in\{1,\ldots,r\}$. Using \eqref{eq:check_f_beta_2}--\eqref{eq:check_f3},
one can write
\begin{equation}\label{eq:vbarbetacheck} {\check
v}_{i+1}^{(\beta)} = \check A_i {\check v}_i^{(\beta)} = \check A_i \cdots
\check A_{\ell+1} {\check v}_{\ell+1}^{(\beta)},\quad i
>\ell.
\end{equation}

Since the vectors $\check h^{(\beta)}_{\ell\ell}$ are linearly independent for any $\beta\in
\{1,\ldots,r\}$, it follows that the vectors ${\check
v}_{\ell+1}^{(\beta)}$ are linearly independent as well, and belong to
$\check V_0$. As a result, $\check d_0 = \dim{\check{V}_{0}} \ge r$. Also,
$\check A_i \un =\un$, so
$\left[\begin{array}{c} \mathbf{1} \\
\mathbf{1}\end{array}\right]\in \check{V}_{0}$.
Then one proceeds as in the proof of Theorem \ref{thm:oseledecfull}, with a few minor changes. \qed

\section{Explicit computations in the reducible case}\label{app:bolthausen}

Set $P_n = Q \Delta_n$, $Q_n = Q(I- \Delta_n)$ and $R_n=0$. Using \citet{Bolthausen/Goldsheid:2000} notations, for a given $a \in \mathbb{Z}$ and a given stochastic matrix $\rho$, for $n>a$,  define
$\psi_n = \psi_{n,a,\rho} = (I-R_n-Q_n\psi_{n-1})^{-1}P_n$, where $\psi_a = \psi_{a,a,\rho} = \rho$. It is not easy to compute $\psi_{n,a}$ for a small $a$, but if the probabilities are periodic with period $3$ for example, then $\psi_{0,-6n} = \psi_{6n,0}$,  $\psi_{0,-6n-1}= \psi_{6n+3,2}$,   $\psi_{0,-6n-2} =\psi_{6n+3,1}$, $\psi_{0,-6n-3} =\psi_{6n+3,0}$, $\psi_{0,-6n-4}  =\psi_{6n+6,2}$ and  $\psi_{0,-6n-5}  =\psi_{6n+6,1}$.

 In \citet[Theorem 1]{Bolthausen/Goldsheid:2000}, the authors claim that the limit $\xi_n = \lim_{a\to -\infty} \psi_{n,a,\rho}$ exists and is independent of $\rho$. One crucial hypothesis for the proof is the existence of only one communication class (a.e.). As noted in Remark \ref{rem:bolthausen}, this is usually not the case here, especially for Game C'.

We show next that \citet[Theorem 1]{Bolthausen/Goldsheid:2000} does not hold for Game C', because either the limit does not exist, or it depends on the initial value $\rho$. In fact, starting from $\rho = \left(\begin{array}{cc} 0 & 1\\ 0 & 1\end{array}\right)$, the limit does not exist.

Now, $\lim_{n\to\infty}\psi_{0,-2n} =
\left(\begin{array}{cc} 0 & 1\\  0.9987  &  0.0013\end{array}\right)$  and $\lim_{n\to\infty}\psi_{0,-2n-1} =
\left(\begin{array}{cc}  0.0039  &  0.9961\\  1  &  0\end{array}\right)$. Next, if one starts from $\rho = Q$, then $\xi_n \equiv Q$.
The above computations ensue from the following set of equations for Game C':
if  $\rho =  \left(\begin{array}{cc}  x_a &  1-x_a\\  1-y_a  &  y_a\end{array}\right)$, then $\psi_n = \left(\begin{array}{cc}  x_n &  1-x_n\\  1-y_n  &  y_n\end{array}\right)$, where,
for any $n\ge a$,
\begin{eqnarray*}
z_{n+1} &=&  p_{n+1}^{(1)} p_{n+1}^{(2)} + p_{n+1}^{(2)} q_{n+1}^{(1)} x_n  + p_{n+1}^{(1)} q_{n+1}^{(2)} y_n, \\
x_{n+1} &=&  p_{n+1}^{(1)} q_{n+1}^{(2)} y_n/z_{n+1},\\
y_{n+1} &=&  p_{n+1}^{(2)} q_{n+1}^{(1)} x_n/z_{n+1}.
\end{eqnarray*}

\section{Computation of $A_2A_1$ in Example \ref{ex:counter}}

It is easy to check that $A_2A_1$ is given by
$$
 \left[\begin{array}{cccc}
1+\rho_1^{(2)}+ \rho_1^{(2)}\rho_2^{(1)}  & 0 & 0 & -\rho_1^{(2)}  -\rho_1^{(2)}\rho_2^{(1)}\\
0 &  1+\rho_1^{(1)}+  \rho_1^{(1)}\rho_2^{(2)} &  -\rho_1^{(1)} -\rho_1^{(1)} \rho_2^{(2)} & 0 \\
0&  1+\rho_1^{(1)}  & -\rho_1^{(1)} & 0 \\
1+\rho_1^{(2)}  & 0 & 0 & -\rho_1^{(2)}\\
\end{array}
\right].
$$
Thus the eigenvalues are $1$,$1$, $\rho_1^{(1)} \rho_2^{(2)}$, and $\rho_1^{(2)} \rho_2^{(1)}$.
\bibliographystyle{apalike}
\def\cprime{$'$}

\end{document}